\documentclass[mnsc,nonblindrev]{oo} 
\OneAndAHalfSpacedXI 

\usepackage{natbib}
 \bibpunct[, ]{(}{)}{,}{a}{}{,}%
 \def\bibfont{\small}%
\TheoremsNumberedThrough     
\ECRepeatTheorems

\EquationsNumberedThrough    

\MANUSCRIPTNO{}

\usepackage[utf8]{inputenc}
\usepackage[english]{babel}
\usepackage[inline]{enumitem}

\usepackage{multirow,tabularx}
\usepackage{booktabs}

\usepackage{algorithm}
\usepackage[noend]{algpseudocode}

\usepackage{amsmath,amsfonts,amssymb,amsbsy}
\usepackage{mathtools,bm}

\usepackage{tikz}
\usepackage{graphicx}
\usepackage{color, xcolor}
\usepackage{subfig}
\usepackage{adjustbox}

\usetikzlibrary{arrows, automata, shapes, decorations.pathreplacing, calligraphy, decorations.markings}

\usepackage{comment}
\usepackage[colorinlistoftodos]{todonotes}

\usepackage[driverfallback=dvipdfm]{hyperref}
\hypersetup{
	colorlinks,
	linkcolor={red!50!black},
	citecolor={blue!50!black},
	urlcolor={blue!80!black}
}

\newcommand{\Exp}{\mathbb{E}}
\newcommand{\R}{\mathbb{R}}

\newcommand{\B}{\{ 0, 1 \}}
\newcommand{\Prob}{\mathbb{P}}

\newcommand{\A}{\mathcal{A}}

\newcommand{\bx}{\bm{x}}
\newcommand{\bz}{\bm{z}}
\newcommand{\by}{\bm{y}}
\newcommand{\bv}{\bm{v}}
\newcommand{\bu}{\bm{u}}
\newcommand{\blambda}{\bm{\lambda}}
\newcommand{\btheta}{\bm{\theta}}
\newcommand{\balpha}{\bm{\alpha}}
\newcommand{\bmu}{\bm{\mu}}

\newcommand{\bxhat}{\hat{\bx}}
\newcommand{\bzhat}{\hat{\bz}}

\newcommand{\xhat}{\hat{x}}
\newcommand{\zhat}{\hat{z}}

\newcommand{\zbar}{\bar{z}}

\newcommand{\Starttt}{\texttt{s}}
\newcommand{\Endtt}{\texttt{e}}

\newcommand{\depots}{\mathcal{K}}
\newcommand{\depotsStart}{\depots^\Starttt}
\newcommand{\depotsEnd}{\depots^\Endtt}
\newcommand{\trips}{\mathcal{I}}
\newcommand{\routes}{\mathcal{R}}
\newcommand{\scenarios}{\mathcal{S}}
\newcommand{\compatible}{\mathcal{C}}

\newcommand{\ccset}{\mathcal{F}}
\newcommand{\mcfset}{\mathcal{X}}
\newcommand{\subproblem}{\mathcal{F}}
\newcommand{\subproblemLP}{\overline{\subproblem}}
\newcommand{\cutset}{\Lambda}

\newcommand{\validset}{\Theta}

\newcommand{\schedule}{\mathcal{V}}
\newcommand{\scheduleIS}{\mathcal{V}^{\texttt{IS}}}
\newcommand{\scheduleMIS}{\mathcal{V}^{\texttt{MIS}}}
\newcommand{\scheduleCMIS}{\mathcal{V}^{\texttt{C-MIS}}}
\newcommand{\scheduleMISExtra}{\mathcal{V}^{\texttt{Extra}}}

\newcommand{\tripsMIS}{\trips^{\texttt{C-MIS}}}
\newcommand{\tripsMISlast}{\tripsMIS_{\texttt{last}}}
\newcommand{\tripsMISfirst}{\tripsMIS_{\texttt{first}}}
\newcommand{\tripsMISmiddle}{\tripsMIS_{\texttt{middle}}}

\newcommand{\scheduleDepot}{\overline{\schedule}}
\newcommand{\scheduleDepotOne}{\scheduleDepot^1}
\newcommand{\scheduleDepotZero}{\scheduleDepot^0}

\newcommand{\numschedule}{V}
\newcommand{\numscheduleDepot}{\overline{\numschedule}}

\newcommand{\bus}{\mathcal{B}}
\newcommand{\numbuses}{B}
\newcommand{\delaytrips}{\mathcal{D}}
\newcommand{\delayMIS}{\delaytrips^{\texttt{C-MIS}}}

\newcommand{\numtrips}{I}
\newcommand{\numdepots}{K}
\newcommand{\numroute}{R}
\newcommand{\numscenarios}{S}
\newcommand{\numarcs}{A}

\newcommand{\cost}{\texttt{c}}

\newcommand{\stime}{\texttt{s}}
\newcommand{\dur}{\texttt{d}}
\newcommand{\travel}{\texttt{t}}
\newcommand{\approxdur}{\tilde{\dur}}
\newcommand{\approxtravel}{\tilde{\travel}}
\newcommand{\express}{\texttt{e}}
\newcommand{\capacity}{\texttt{b}}
\newcommand{\loc}{\texttt{l}}
\newcommand{\locstart}{\loc^\Starttt}
\newcommand{\locend}{\loc^\Endtt}
\newcommand{\ub}{\texttt{ub}}
\newcommand{\lb}{\texttt{lb}}

\newcommand{\bigM}{\texttt{M}}
\newcommand{\bigMstart}{\bigM^{\texttt{start}}}
\newcommand{\bigMotp}{\bigM^{\texttt{OTP}}}

\newcommand{\s}{s}

\newcommand{\graph}{\mathcal{G}}
\newcommand{\nodes}{\mathcal{N}}
\newcommand{\arcs}{\mathcal{A}}

\newcommand{\ccprob}{\epsilon}
\newcommand{\serviceRoute}{\delta^{\texttt{route}}}
\newcommand{\serviceTrips}{\delta^{\texttt{trip}}}
\newcommand{\numserviceTrip}{\texttt{f}^{\texttt{trip}}}
\newcommand{\numserviceRoute}{\texttt{f}^{\texttt{route}}}

\newcommand{\partitions}{\mathcal{P}}
\newcommand{\numpartition}{P}
\newcommand{\tripgroups}{\trips^{\texttt{gr}}}
\newcommand{\groupsize}{\texttt{m}^{\texttt{gr}}}
\newcommand{\lagrsub}{\textsf{L}}
\newcommand{\lagrconst}{\texttt{C}}
\newcommand{\grad}{g}

\newcommand{\dualtrip}{\sigma^{\texttt{trip}}}
\newcommand{\dualroute}{\sigma^{\texttt{route}}}
\newcommand{\dualMIS}{\sigma^{\texttt{MIS}}}
\newcommand{\dualstart}{\alpha}
\newcommand{\dualOTP}{\pi}
\newcommand{\duallb}{\beta}

\newcommand{\dualtriphat}{\hat{\sigma}^{\texttt{trip}}}
\newcommand{\dualroutehat}{\hat{\sigma}^{\texttt{route}}}
\newcommand{\dualMIShat}{\hat{\sigma}^{\texttt{MIS}}}
\newcommand{\dualstarthat}{\hat{\dualstart}}
\newcommand{\dualOTPhat}{\hat{\dualOTP}}
\newcommand{\duallbhat}{\hat{\duallb}}

\newcommand{\algReturn}{\textbf{return}}

\newcommand{\algGreedy}{\textsc{Greedy}}

\newcommand{\algMIS}{\textsc{C-MIS\_Greedy}}
\newcommand{\algDelayToExplain}{\texttt{start\_time\_to\_explain}}
\newcommand{\tol}{\epsilon^{\texttt{tol}}}
\newcommand{\con}{\texttt{con}}

\newcommand{\algPartition}{\textsc{PartitionTrips}}

\newcommand{\misConflict}{\textsf{MIS}}
\newcommand{\misGreedy}{\textsf{C-MIS}}
\newcommand{\misGreedyEnhance}{\textsf{EC-MIS}}

\newcommand{\none}{\textsf{Nn}}
\newcommand{\both}{\textsf{Bo}}
\newcommand{\validInequalities}{\textsf{VI}}
\newcommand{\zcont}{\textsf{ZC}}

\newcommand{\mean}{\textsf{Mean}}
\newcommand{\percentile}{\textsf{75Per}}
\newcommand{\chanceSAA}{\textsf{CC}}
\newcommand{\La}[1]{\textsf{La#1}}

\tikzstyle{trip} = [draw,line width=1.1pt, blue!80!black, postaction={decorate}]
\tikzstyle{dtrip} = [draw, line width=0.4pt, red, postaction={decorate}]
\tikzstyle{depot} = [blue, fill=blue]

\tikzstyle{node trip} = [circle, draw = black, font=\scriptsize, inner sep=3pt]
\tikzstyle{node trip delayed} = [circle, draw = black, fill=gray!20!white, font=\scriptsize, inner sep=3pt]
\tikzstyle{node depot} = [circle, draw = gray, font=\scriptsize, inner sep=3pt]
\tikzstyle{arc} = [draw,line width=0.5pt,->, black]
\tikzstyle{pull arc} = [draw, dashed, line width=0.5pt,->, black]

\tikzstyle{tw} = [font=\scriptsize]

\tikzstyle{arc sol} = [arc, line width=1pt,
preaction={
	draw,gray!30!white,-,
	double=gray!30!white,
	double distance=3pt,
}]
\tikzstyle{pull arc sol} = [pull arc, line width=1pt,
preaction={
	draw,gray!30!white,-,
	double=gray!30!white,
	double distance=3pt,
}]

\begin{document}



\RUNTITLE{Incorporating Service Reliability in Multi-depot Vehicle Scheduling}

\TITLE{Incorporating Service Reliability in Multi-depot Vehicle Scheduling: A Chance-Constrained Approach}

\ARTICLEAUTHORS{%
\AUTHOR{Margarita P. Castro}
\AFF{Department of Industrial and Systems Engineering, Pontificia Universidad Católica de Chile, Santiago 7820436, Chile \\ \EMAIL{margarita.castro@ing.puc.cl}} 
\AUTHOR{Merve Bodur}
\AFF{School of Mathematics, University of Edinburgh, Edinburgh  EH9 3FD, UK, \\ \EMAIL{merve.bodur@ed.ac.uk}}
\AUTHOR{Amer Shalaby}
\AFF{Department of Civil and Mineral Engineering, University of Toronto, Toronto Ontario M5S 1A4, Canada, \\\EMAIL{amer.shalaby@utoronto.ca}}
} 

\ABSTRACT{%
The multi-depot vehicle scheduling problem (MDVSP) is a critical planning challenge for transit agencies. We introduce a novel approach to MDVSP by incorporating service reliability through chance-constrained programming (CCP), targeting the pivotal issue of travel time uncertainty and its impact on transit service quality. Our model guarantees service reliability 
measured by on-time performance (OTP), a primary metric for transit agencies, and fairness across different service areas. 
We propose an exact branch-and-cut (B\&C) scheme to solve our CCP model. We present several cut-generation procedures that exploit the underlying problem structure and analyze the relationship between the obtained cut families. Additionally, we design a Lagrangian-based heuristic to handle large-scale instances reflective of real-world transit operations. Our approach partitions the set of trips, each subset leading to a subproblem that can be efficiently solved with our B\&C algorithm, and then employs a procedure to combine the subproblem solutions to create a vehicle schedule that satisfies all the planning constraints of the MDVSP. Our empirical evaluation demonstrates the superiority of our stochastic variant in achieving cost-effective schedules with reliable OTP guarantees compared to alternatives commonly used by practitioners, as well as the computational benefits of our methodologies. 
}%

\maketitle

%



\section{Introduction}

Transit agencies face several decision-making problems to provide a high-quality service to their customers at a low cost, such as timetabling, crew scheduling, and vehicle scheduling \citep{ibarra2015planning}. In particular, vehicle scheduling problems constitute one of the most important problem classes to public transportation companies since they directly impact the service quality (e.g., waiting time of passengers) and the operational costs (i.e., fleet size and vehicle usage). These problems involve assigning timetabled trips to buses stationed at different depots to minimize the total operational costs.  
One of the most studied problems in the vehicle scheduling literature is the multi-depot vehicle scheduling problem (MDVSP). Its classic variant considers a fleet of homogeneous buses and multiple depots. Each bus is stationed at a particular depot (i.e., its starting location) and has to return to that depot at the end of the day. Each bus is assigned to a sequence of timetabled trips, where each trip has a start and end location, a scheduled start time, and a duration. Since trips can start and end in different locations, buses might travel empty from one location to the next, referred to as deadhead trips. The goal of the MDVSP is to find a vehicle schedule (i.e., trip assignments to buses) that minimizes the total cost given by the deployed fleet size (i.e., number of buses employed) and the operational cost associated with deadhead trips (e.g., fuel or electricity). In practice, a vehicle schedule corresponds to a single day of operation, which will be employed for several weeks (usually 2 to 3 months) given the seasonal demand fluctuations. 
    
The classic MDVSP has been vastly studied in the literature, given its relevance to practitioners \citep{ibarra2015planning} and its computational complexity \citep{bertossi1987some}. There are several problem variants in the literature, including extensions that consider fleets with electric vehicles \citep{fusco2013model,chao2013optimizing,tang2019robust}, and integrating the MDVSP with other transit problems such as crew scheduling \citep{mesquita2013decomposition} and timetabling \citep{liu2007regional}. However, only few works aim to address one of the main practical challenges of the MDVSP: the uncertainty in the travel times. Among these, the majority either considers dynamically changing the vehicle schedule in the operational stage \citep{guedes2018real,he2018vehicle}, which might be highly impractical, or studies simple stochastic variants that penalize expected delays thus obtaining deterministic models \citep{naumann2011stochastic,shen2016probabilistic, ricard2022increasing}, also ignoring the proper modeling of delay propagation among trips and their impact on service quality requirements except \citep{ricard2022increasing}. 
We further discuss these related works in detail in Section \ref{sec:literature}. 

On the other hand, practitioners usually rely on deterministic approaches where they overestimate travel times to hedge against possible delays, which results in vehicle schedules with higher costs and no guarantee on service quality or reliability. Agencies most commonly use the so-called on-time performance (OTP) measure for reliability. According to this metric, a bus is considered to be on time when it departs (or alternatively arrives) a predetermined bus stop within a certain range of its scheduled departure (or arrival) \citep{european2002transportation}, and OTP corresponds to the percentage of departing on time. The on-time range choice varies across the transit industry, the most commonly used one having been around 1-minute earlier and 5-minute later than the scheduled time \citep{guenthner1988distribution}. The 2018 data is summarized in \citep{TransitCenter} provides the ranges used by top 20 transit agencies in the United States and their weekday OTP ranging from 44\% to 75\%. In that regard, Figure \ref{fig:TTCandTfL} presents the recent information for Toronto, Ontario, Canada and London, United Kingdom. \cite{TTC} uses 1-minute early and 5-minute late for the on-time range, and releases a daily report on their website for OTP targets for their services and what is achieved (Figure \ref{subfig:TTC}). \cite{TfL} uses 2-minute early and 5-minute late for the OTP range, and releases a quarterly report on their bus services (Figure \ref{subfig:TfL}). 

\begin{figure}[t]
\centering
    \subfloat[\centering TTC Daily Customer Service Report on May 15, 2024.]{{\includegraphics[width=6.5cm]{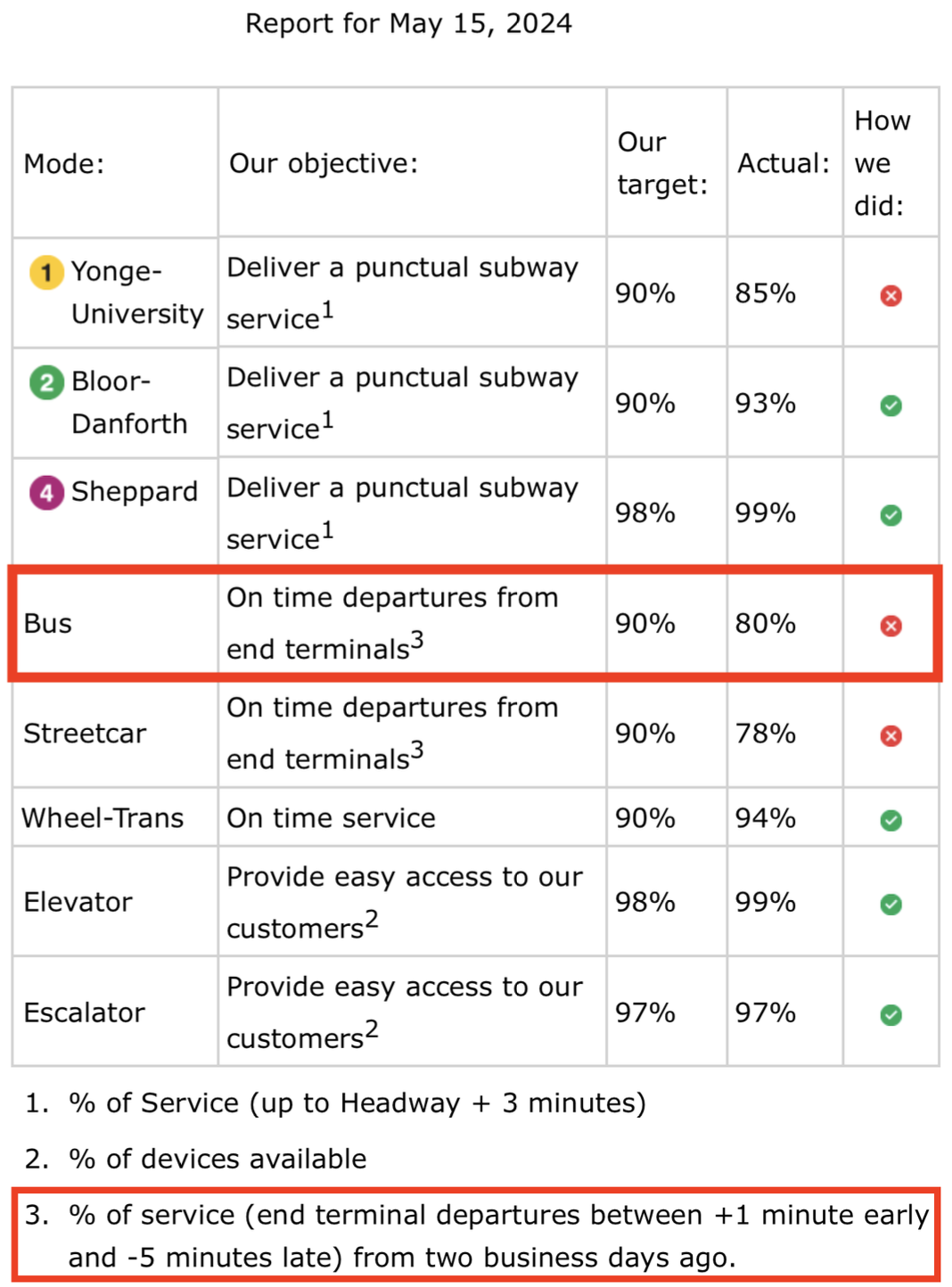} }
    \label{subfig:TTC}}
    \qquad
    \subfloat[\centering TfL Route Results for London Bus Services in 
Quarter 04 23/24.]{{\includegraphics[width=8cm]{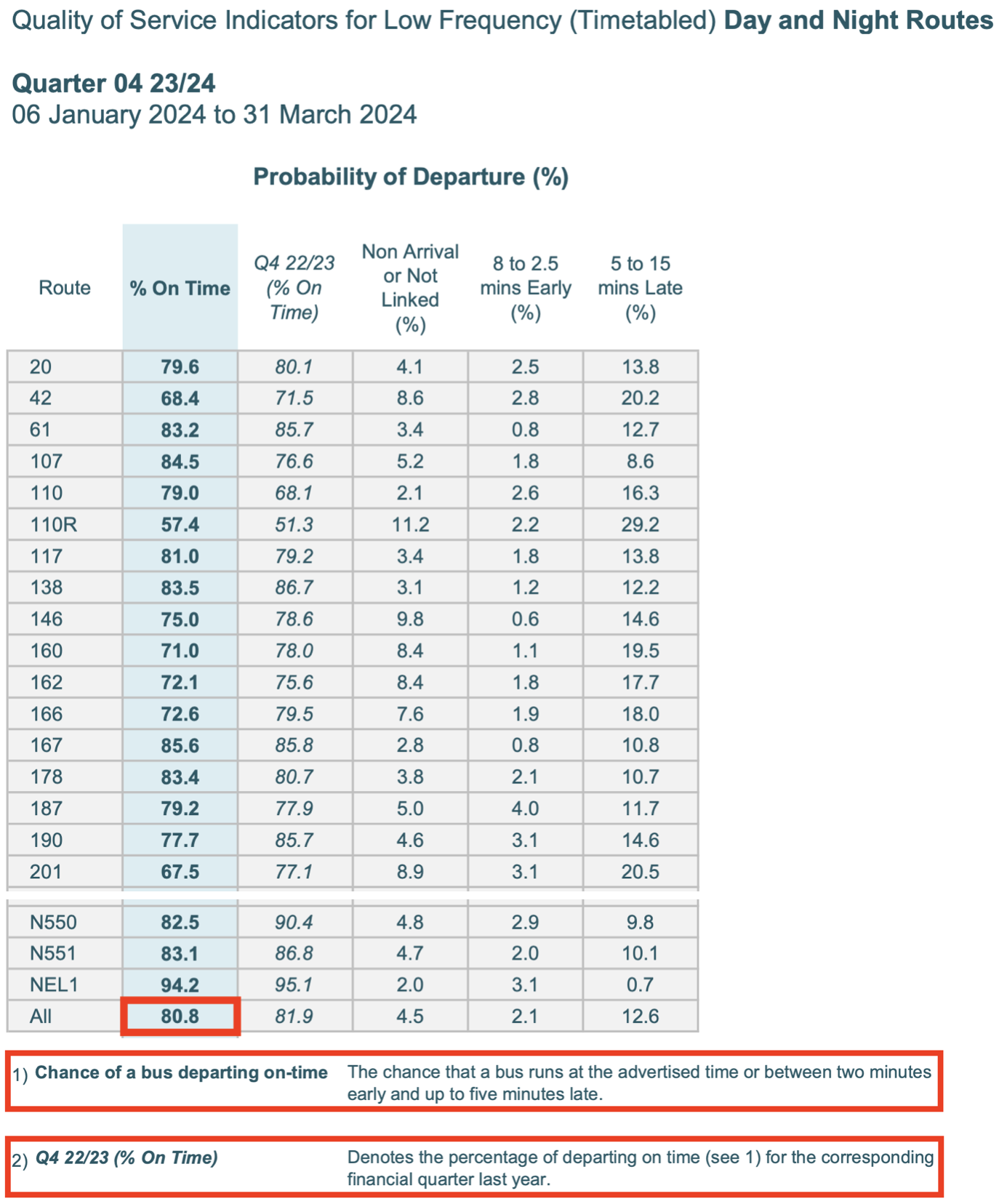} }
\label{subfig:TfL}}
    \caption{On-time performance statistics from Toronto and London bus systems.}%
    \label{fig:TTCandTfL}%
\end{figure}

Motivated by the practical importance of MDVSP and the limitations of the literature, we propose a novel stochastic MDVSP variant that considers stochastic travel times to achieve the desired service requirements defined based on OTP. The main goal is to build a cost-optimal schedule that satisfies the service requirements given by OTP. In that regard, we not only ensure that most trips have to start on time but also incorporate  fairness considerations into schedule  
(e.g., services on different areas must have similar OTP). Moreover, we consider a risk-averse setting to fulfill these service requirements given the high operational reliability and user satisfaction that transit agencies aim for in practice. Thus, we model the problem as a chance-constrained programming (CCP) model, which enforces the service requirements up to a certain probability threshold.  From now own, we call this problem the chance-constrained MDVSP (CC-MDVSP). 
    
We propose exact and heuristic methodologies to address the scenario reformulation of CC-MDVSP.   
Our exact procedure is a branch-and-cut (B\&C) scheme that iteratively offers a cost-efficient vehicle schedule and evaluates the service conditions in a set of scenarios given by different travel time realizations. We present several cut generation algorithms based on different problem characteristics and analyze their theoretical properties. In particular, we show that our subproblems can be solved via a greedy algorithm in polynomial time, and more importantly their solutions can be used in an iterative procedure to efficiently identify minimal infeasible subsystems (MISs) and accordingly cuts, which we call constraint-based MIS cuts, stronger than the traditionally used alternatives. 
Moreover, we design a heuristic procedure based on our exact methodology to handle realistic large-scale problems. More specifically, we consider a Lagrangian scheme that decomposes the problem into sub-problems with a smaller set of trips, which can be efficiently solved with our exact B\&C methodology. We then solve the Lagrangian dual problem that enforces the service requirements across all trips to find a cost-efficient solution for the complete problem. 
    
We perform extensive computational experiments which 
show that our proposed methodologies return cost-efficient vehicle schedules that fulfill the service requirements. 
Comparisons with those commonly used by practitioners (i.e.,  deterministic models with over-estimations of the travel times) 
illustrate that our techniques find lower-cost schedules with reliable service guarantees. 
    
The remainder of the paper is organized as follows. Section \ref{sec:literature} provides a literature review on the MDVSP, its stochastic variants, and related CCP problems. Section \ref{sec:problem} describes our novel stochastic MDVSP variant with service requirements and presents our CCP model. Section \ref{sec:decomposition} introduces our scenario reformulation and or the B\&C  decomposition approach. Section \ref{sec:cutgeneation} presents several cut generation alternatives, methodological enhancements, and theoretical results. Section \ref{sec:large-scale} describes our Lagrangian-based scheme to handle large-scale problems. Lastly, Section \ref{sec:experiments} presents the numerical experiment results and we end with some concluding remarks in Section \ref{sec:conclusions}.
\section{Literature Review} \label{sec:literature}

We review the two main bodies of literature that are closely related to our work. Section \ref{sec:lit-mdvsp} describes relevant works in the MDVSP literature, emphasizing existing exact algorithms for the deterministic and stochastic variants of this problem. Section \ref{sec:lit-cc} reviews relevant works on CCP, especially ones that handle chance constraints with integer variables. 

\subsection{Multi Depot Vehicle Scheduling and Extensions} 
\label{sec:lit-mdvsp}


The deterministic MDVSP is 
well-studied due to its practical application to public transit agencies \citep{ibarra2015planning} and its combinatorial structure which makes the problem NP-hard for the case with two or more depots \citep{bertossi1987some}. Several models and optimization techniques have been proposed to tackle this challenging problem (see, e.g.,  \cite{bunte2009overview, pepin2009comparison}). \cite{carpaneto1989branch} propose the first exact algorithm for MDVSP based on the branch-and-bound procedure and a single-commodity model with sub-tour breaking constraints, which was then improved by \cite{fischetti2001polyhedral} via a branch-and-cut approach.  

One of the most used models to solve the MDVSP is the time-space formulation \citep{kliewer2006time} that discretizes time to consider each trip's start and end time points. This formulation tends to be quite large in the number of variables, but can be 
solved using a branch-and-price (B\&P) algorithm. \cite{kulkarni2018new} present an inventory formulation that also relies on a time-space network. A column generation approach is proposed to solve the linear programming relaxation along with a heuristic procedure to find high-quality integer solutions. 

Alternatively, \cite{desrosiers1995time} propose a multi-commodity flow (MCF) model with a cubic number of variables concerning the number of trips and depots. \cite{hadjar2006branch} studied the MCF polyhedral structure and proposed several enhancements, including a reduced cost-fixing strategy and a branch-and-cut approach to remove fractional solutions via odd-cycle cuts. We note that the MCF model is more amenable to our variant with stochastic travel times than the time-space formulations because the former implicitly considers time when computing compatible trips, while the latter explicitly uses time points to create the model. 

One of the main practical limitations of the deterministic MDVSP is that it ignores the stochastic nature of travel times and, thus, its delayed consequences. Since delays and disruptions in service are a major concern for transit agencies, several works have built deterministic contingency models to fix the schedule during an operational day, such as rescheduling \citep{li2007vehicle,guedes2018real}, real-time recovery \citep{visentini2014review}, damaging disruptions \citep{uccar2017managing} and trip-shifting strategies \citep{desfontaines2018multiple}, among others. However, these techniques require real-time data and a coordinator to make timely decisions, which are very costly for transit agencies, especially if they have to be made most days. 

Surprisingly, only few works consider stochastic travel times while creating the vehicle schedule to reduce the need for such contingency options during operation. \cite{naumann2011stochastic} base their MDVSP formulation on 
the time-space model of \cite{kliewer2006time}. 
In their approach, delay values for each trip (except depot and deadhead trips) are generated from an exponential distribution to construct a set of delay scenarios.  
The idea is to add penalties (i.e., additional cost) to the waiting arcs in the network, which connect pairs of trips and have an associated predetermined waiting time. More specifically, given a consecutive trip and waiting arc combination, if the sampled delay of the trip leads to a delay in the subsequent trip, then a quadratic penalty cost is added to the waiting arc cost. Despite having delay scenarios, the resulting formulation becomes a deterministic model where expected delay penalty costs are added to the waiting arcs. 
\cite{shen2016probabilistic} and \cite{shen2017vehicle}  follow a similar approach, except base their formulation on the MCF model, ending up with a deterministic model where additional costs are added to the trip arcs. Therefore, these approaches fall short in properly modeling the  
propagation effects of the delays.

Since the delay propagation for consecutive trips can have a massive impact in practice, recent works have considered dynamic approaches that build and repair the schedule during the operational day. \cite{he2018vehicle} propose an approximated dynamic programming approach to build schedules daily that minimize cost and reduce delays of consecutive trips. \cite{tang2019robust} introduce a dynamic and static model that incorporates the need to charge a fleet of electric buses. The static model uses a buffer-distance strategy to protect from delays, while the dynamic model periodically reschedules the bus fleet during a day's operations. We note that none of these approaches actually create a robust schedule that performs well in practice without the need to reschedule (i.e., return a reliable schedule). 

To the best of our knowledge, the recent work by \cite{ricard2022increasing} is the only one that aims to create a reliable vehicle schedule using stochastic travel times and consider proper delay propagation. A set partition model is proposed, which enumerates all possible trip sequences, thus is amenable to 
B\&P. Each sequence considers the operation cost plus a delay penalty,  calculated using discretized probability mass function to account for delay propagation \citep{ricard2022predicting}. As in the previously reviewed works, the resulting model is deterministic since the stochastic travel times are only used to compute the expected delays, which is then transformed into a penalty in the objective function. Moreover, their model is sensitive to the weight assigned to the delay penalty cost, and, as such, users need to tune this parameter to obtain schedules with the desired reliability. In contrast, our CC-MDVSP variant includes the service reliability conditions directly in the model via a CCP that optimizes for a risk-averse setting (i.e., ensure for instance 95\% of the days with a reliable service, along with fairness among routes). Also, we measure reliability based on the number of trips that start later than their scheduled start time as commonly done in transit agencies (e.g., in \cite{TTC, TfL}), while \cite{ricard2022increasing} consider metrics focused on the public transit user's perspective.





\subsection{Chance Constrained Programming} \label{sec:lit-cc}

CCP is a well-studied modeling paradigm in stochastic programming that considers stochastic constraints to be satisfied with a certain probability \citep{charnes1959chance}. While these constraints can be cast as conic constraints for parameters with Gaussian or similar distribution \citep{kuccukyavuz2021chance}, handling other types of distribution usually requires some sort of sampling and scenario reformulation \citep{pagnoncelli2009sample}. 

\cite{ruszczynski2002probabilistic} first introduced the scenario reformulation for CCP when considering a discrete distribution. This reformulation replaces the random variables with a set of scenario and binary variables, making the problem non-convex and significantly increasing its size. This procedure can be extended for probability distributions with infinite support using sample average approximation (SAA), i.e., reformulate the problem with a proper number of scenarios and obtain statistically valid lower and upper bounds to the original problem \citep{luedtke2008sample}.

While the number of scenarios can be reduced using SAA, the resulting model is still a large-scale mixed-integer linear program (MILP) with a weak linear programming (LP) relaxation given by a large set of big-M constraints. \cite{luedtke2010integer} study the structure of the problem when considering only right-hand-side uncertainty and propose an improved relaxation for the case with only continuous variables. There is also significant research on valid inequalities mechanism to strengthen the big-M coefficients (see, e.g., \cite{kuccukyavuz2012mixing,song2014chance,ahmed2018relaxations}). 

One of the most common approaches to deal with the scenario-based MILP is B\&C  \citep{luedtke2014branch}. The procedure decomposes the problem into a master problem containing all the deterministic constraints and one subproblem for each scenario. Given a fixed candidate solution, the B\&C  solves the subproblems for each scenario and returns a cut if the solution violates the CC feasibility set for that scenario. Although general, this procedure is usually tailored for each application depending on the structure of the subproblem and the nature of its variables. There are few applications in the literature that consider CCP problems with integer variables, such as scheduling \citep{deng2016decomposition}, vehicle routing \citep{dinh2018exact} and partial set covering \citep{wu2019probabilistic}. These works make problem-specific variations of the B\&C decomposition, and all their integer variables are part of the master problem. In contrast, our problem considers integer variables in both the master and subproblem, a case with scarce literature. Thus, we propose specialized procedures to generate cuts given the subproblem non-convexity. 

\cite{canessa2019algorithm} propose an algorithm to solve one-stage pure binary linear CCPs using irreducibly infeasible subsystems (IIS).
In a B\&C framework, they solve the restriction of the problem where all scenario constraints are enforced (except those corresponding to the scenarios that are eliminated by the branching decisions, if any). When such a restriction is infeasible, they generate an IIS using a commercial solver and generate a cut on the scenario variables enforcing that at least one among the identified minimal set should be eliminated. They can generate a similar IIS cut also in the case where the restricted problem is feasible, by adding a constraint bounding the optimal value and turning the restricted problem into one looking for an improving feasible solution. Being only on the scenario variables, their cuts do not capture   any relationship between the scenario variables and the original variables of the model. 
In contrast, to solve a two-stage CCP with integer recourse, our proposed methodology relies on a B\&C procedure where we identify IIS (i.e., MIS) based on the recourse (i.e., second-stage) problems and generate cuts linking the original first-stage variables (i.e., vehicle scheduling variables) and the scenario variables. 
Moreover, leveraging our problem structure, we introduce a polynomial procedure to identify IIS as well as some cut strengthening procedures. Our newly proposed  constraint-based MIS (C-MIS) and extended C-MIS cuts can be useful for some other applications.

\section{Problem Description and Formulation} \label{sec:problem}

We now describe the proposed stochastic variant of the MDVSP that ensures service requirements in a risk-averse setting. In what follows, we use calligraphic font for sets, uppercase letters for the sets' cardinality, and lowercase letters for the sets' elements if possible (e.g., $a\in \A$ is an element of set $\A$ with cardinality $|\A|=A$), and typewriter font for parameters (e.g., $\texttt{b}$ is a parameter for the problem). Also, we use $\Prob(\cdot)$ as the probability operator and $\Exp(\cdot)$ as the expected value operator. 

The MDVSP considers a set of $\numtrips$ timetabled trips, $\trips=\{1,...,\numtrips\}$, and a set of $\numdepots$ depots, $\depots=\{1,...,\numdepots\}$.  Each depot $k \in \depots$ has a capacity $\capacity_k$ that represents the maximum number of buses that can be placed at depot $k$. As in the standard deterministic variant (e.g., in \cite{carpaneto1989branch}), we consider a homogeneous bus fleet (i.e., all buses are the same), so a timetabled trip $i \in \trips$ can be assigned to any bus which can be associated with any depot $k \in \depots$. Buses must start and end their tour (i.e., sequence of their assigned timetabled trips) at their associated depot.

Let $\xi$ represent the underlying stochastic process associated with this problem; that is, $\xi$ is the set of random variables representing bus travel times. Each timetabled trip $i \in \trips$ has a scheduled start time $\stime_i$, a start location $\locstart_i$, an end location $\locend_i$, and a stochastic duration $\dur_i(\xi)$ that represents the time to go from its start location $\locstart_i$ to its end location $\locend_i$. 

In addition to the set of timetabled trips $\trips$, we define deadhead trips as bus movements to reach the start location of a trip or return to its associated depot. Specifically, a bus associated with depot $k \in \depots$ performs a deadhead trip if it travels from: 
\begin{enumerate*}[label=(\roman*)]
    \item $k$ to the start location of trip $i \in \trips$,
    \item the end location of $i$ to the start location of $j\in \trips$ ($i \neq j$), or  
    \item the end of location $i$ to depot $k$.
\end{enumerate*}
Then, the random variables $\travel_{ij}(\xi)$, $\travel_{ki}(\xi)$, and $\travel_{ik}(\xi)$ represent the stochastic travel time of deadhead trips between timetabled trips $i,j \in \trips$ and from/to the depot $k \in \depots$, accordingly.

Our CC-MDVSP variant seeks a feasible assignment of timetabled trips to buses such that each trip is assigned to a single bus and the scheduled starting time is met as closely as possible. Thus, each bus assignment corresponds to a sequence of timetabled trips such that there is one deadhead trip between each timetabled trip and one deadhead trip to leave and return to the depot. Then, a solution is a vehicle schedule representing the sequence of timetabled trips for each deployed bus. Since we consider a homogeneous bus fleet, it is sufficient to associate trip sequences with depots. 

The objective of the problem is to minimize operational costs.
Given that all timetabled trips are assigned to buses, operational costs are associated with deadhead trips. In that regard, $\cost_{ij}\geq 0$ for trips $i,j \in \trips$ corresponds to the operational cost of scheduling trip $i$ right before trip $j$ on the same bus, which represents, for instance, the fuel 
consumption of the deadhead trip. Analogously, $\cost_{ki}\geq 0$ and $\cost_{ik}\geq 0$ represent the cost of scheduling $i\in \trips$ as the first and last trip of a bus associated with depot $k\in \depots$, respectively. We note that these depot-related deadhead trips typically include the cost of deploying a bus (e.g., the driver's salary) and the operational cost associated with the travel distance (e.g., fuel cost).
Example \ref{exa:mdvsp_instance} presents a small instance of the deterministic MDVSP that illustrates the previously described concepts. 

\begin{example} \label{exa:mdvsp_instance}
    Consider a deterministic MDVSP instance depicted in Figure \ref{fig:example-grid}. Each $10\times7$ grid corresponds to the same MDVSP instance with two depots, $k_1, k_2\in \depots$, positioned at coordinates (1,2) and (8,3), respectively. Each thick blue line corresponds to a timetabled trip, where the arrow points towards the end location and the number right next to it represents its index $ i \in \trips = \{1, \dots, 8\}$. For example, trip 1 starts at location $\locstart_1 = (2,1)$ and ends in $\locend_1 = (2,6)$.  Red thin lines correspond to deadhead trips for a specific vehicle schedule. For instance, the line connecting $k_1$ and $\locstart_1$ in both grids corresponds to the deadhead trips from the first depot to the start of trip 1.

    The timetable in Figure  \ref{fig:example-grid} shows the start time and average duration of each trip measured in units of time. In the deterministic setting, we consider each cell distance to represent one unit of time and one unit of operational cost. Moreover, deadhead trips associated with depots have an additional fixed cost of 2, representing bus deployment. For example, the deadhead trip from $k_1$ to trip 1 has an average duration $\bar{t}_{k_1, 1} = 2$ and cost $c_{k_1,1} = 4$.  Both grids represent feasible vehicle schedules for the deterministic setting with average times; that is, each trip $ i \in \trips$ can start at the required $\stime_i$ time. Each solution employs two buses (one per depot), and both solutions have the same total cost of 20 (16 units of operational cost and 8 units for deploying the 2 buses). \hfill $\square$
\end{example}

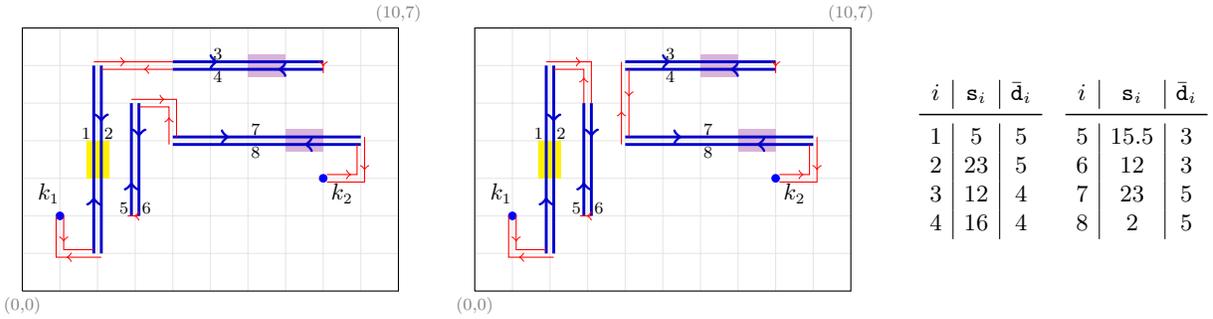
\begin{figure}[htb]
    \centering
    \begin{minipage}{0.75\textwidth}
        \centering
        \begin{tikzpicture}[scale=0.5, every node/.style={scale=0.8}]
    \draw[step=1,gray!20!white,thin] (0,0) grid (10,7);
    \draw[black] (0,0) -- (0,7) -- (10,7) -- (10,0) -- (0,0); 

    \node (00) at (0, -0.4) {\scriptsize \color{gray} (0,0)};
    \node (107) at (10, 7.4) {\scriptsize \color{gray} (10,7)};
    \draw[blue, fill=blue] (1,2) circle (0.1); 
    \node (k1) at (0.7, 2.6) {$k_1$};
    \draw[blue, fill=blue] (8,3) circle (0.1); 
    \node (k2) at (8.5, 2.6) {$k_2$};

    \fill[yellow] (1.7,3) -- (1.7,4) -- (2.3,4) -- (2.3,3); 
    
    \fill[violet!30!white] (6,5.7) -- (7,5.7) -- (7,6.3) -- (6,6.3); 
    \fill[violet!30!white] (7,3.7) -- (8,3.7) -- (8,4.3) -- (7,4.3); 

    \begin{scope}[ decoration={markings, mark=at position 0.3 with {\arrow{>}}} ] 
        \draw[trip] (1.9,1) -- (1.9,6); 
        \draw[trip] (2.1,6) -- (2.1,1);   
        \node (i1) at (1.7, 4.2) {\scriptsize 1};
        \node (i2) at (2.3, 4.2) {\scriptsize 2};

        \draw[trip] (4,6.1) -- (8,6.1); 
        \draw[trip] (8,5.9) -- (4,5.9);   
        \node (i3) at (5.2, 6.3) {\scriptsize 3};
        \node (i4) at (5.2, 5.7) {\scriptsize 4};

        \draw[trip] (4,4.1) -- (9,4.1); 
        \draw[trip] (9,3.9) -- (4,3.9);   
        \node (i3) at (6.2, 4.3) {\scriptsize 7};
        \node (i4) at (6.2, 3.7) {\scriptsize 8};

        \draw[trip] (2.9,2) -- (2.9,5); 
        \draw[trip] (3.1,5) -- (3.1,2);   
        \node (i5) at (2.7, 2.2) {\scriptsize 5};
        \node (i6) at (3.3, 2.2) {\scriptsize 6};
    \end{scope}

    \begin{scope}[ decoration={markings, mark=at position 0.4 with {\arrow{>}}} ] 
        \draw[dtrip] (1.1,2) -- (1.1,1.1) -- (1.9,1.1) ; 
        \draw[dtrip] (2.1,0.9) -- (0.9,0.9) -- (0.9,2) ; 

        \draw[dtrip] (1.9,6.1) -- (4,6.1) ; 
        \draw[dtrip] (4,5.9) -- (2.1,5.9) ; 
        \draw[dtrip] (8,6.1) -- (8,5.8); 

        \draw[dtrip] (8.1,3.1) -- (8.9,3.1) -- (8.9,3.9) ; 
        \draw[dtrip] (9.1,4.1) -- (9.1,2.9) -- (8.1,2.9) ; 

        \draw[dtrip] (3.9,3.9) -- (3.9,4.9) -- (3.1,4.9) ; 
        \draw[dtrip] (2.9,5.1) -- (4.1,5.1) -- (4.1,4.1) ; 
        \draw[dtrip] (3.1,2) -- (2.8,2); 
    \end{scope}
    
\end{tikzpicture}
        \begin{tikzpicture}[scale=0.5, every node/.style={scale=0.8}]
    \draw[step=1,gray!20!white,thin] (0,0) grid (10,7);
    \draw[black] (0,0) -- (0,7) -- (10,7) -- (10,0) -- (0,0); 

    \node (00) at (0, -0.4) {\scriptsize \color{gray} (0,0)};
    \node (107) at (10, 7.4) {\scriptsize \color{gray} (10,7)};
    
    \draw[blue, fill=blue] (1,2) circle (0.1); 
    \node (k1) at (0.7, 2.6) {$k_1$};
    \draw[blue, fill=blue] (8,3) circle (0.1); 
    \node (k2) at (8.5, 2.6) {$k_2$};

    \fill[yellow] (1.7,3) -- (1.7,4) -- (2.3,4) -- (2.3,3); 
    
    \fill[violet!30!white] (6,5.7) -- (7,5.7) -- (7,6.3) -- (6,6.3); 
    \fill[violet!30!white] (7,3.7) -- (8,3.7) -- (8,4.3) -- (7,4.3); 
    
    \begin{scope}[ decoration={markings, mark=at position 0.3 with {\arrow{>}}} ] 
        \draw[trip] (1.9,1) -- (1.9,6); 
        \draw[trip] (2.1,6) -- (2.1,1);   
        \node (i1) at (1.7, 4.2) {\scriptsize 1};
        \node (i2) at (2.3, 4.2) {\scriptsize 2};

        \draw[trip] (4,6.1) -- (8,6.1); 
        \draw[trip] (8,5.9) -- (4,5.9);   
        \node (i3) at (5.2, 6.3) {\scriptsize 3};
        \node (i4) at (5.2, 5.7) {\scriptsize 4};

        \draw[trip] (4,4.1) -- (9,4.1); 
        \draw[trip] (9,3.9) -- (4,3.9);   
        \node (i3) at (6.2, 4.3) {\scriptsize 7};
        \node (i4) at (6.2, 3.7) {\scriptsize 8};

        \draw[trip] (2.9,2) -- (2.9,5); 
        \draw[trip] (3.1,5) -- (3.1,2);   
        \node (i5) at (2.7, 2.2) {\scriptsize 5};
        \node (i6) at (3.3, 2.2) {\scriptsize 6};
    \end{scope}

    \begin{scope}[ decoration={markings, mark=at position 0.4 with {\arrow{>}}} ] 
        \draw[dtrip] (1.1,2) -- (1.1,1.1) -- (1.9,1.1) ; 
        \draw[dtrip] (2.1,0.9) -- (0.9,0.9) -- (0.9,2) ; 

        \draw[dtrip] (3.9,3.9) -- (3.9,6.1); 
        \draw[dtrip] (8,6.1) -- (8,5.8); 
        \draw[dtrip] (4.1,5.9) -- (4.1,4.1); 

        \draw[dtrip] (8.1,3.1) -- (8.9,3.1) -- (8.9,3.9) ; 
        \draw[dtrip] (9.1,4.1) -- (9.1,2.9) -- (8.1,2.9) ; 

        \draw[dtrip] (1.9,6.1) -- (3.1,6.1) -- (3.1,5) ; 
        \draw[dtrip] (2.9,5) -- (2.9,5.9) -- (2.1,5.9) ; 
        \draw[dtrip] (3.1,2) -- (2.8,2); 
    \end{scope}
    
\end{tikzpicture}
    \end{minipage}
    \begin{minipage}{0.11\textwidth}
    \footnotesize	
        \begin{tabular}{c|c|c}
    $i$     & $\stime_i$ & $\bar{\dur}_i$ \\
    \midrule
    1     & 5    & 5 \\
    2     & 23   & 5 \\
    3     & 12   & 4 \\
    4     & 16   & 4 \\
\end{tabular}%

    \end{minipage}
    \begin{minipage}{0.11\textwidth}
    \footnotesize	
        \begin{tabular}{c|c|c}
    $i$     & $\stime_i$ & $\bar{\dur}_i$ \\
    \midrule
    5     & 15.5  & 3 \\
    6     & 12   & 3 \\
    7     & 23    & 5 \\
    8     & 2     & 5 \\
\end{tabular}%

    \end{minipage}
    \caption{Example for the deterministic and CC-MDVSP. Both grids represent the same problem instance with different vehicle schedules each. The right-hand side shows the average timetable information.   }
    \label{fig:example-grid}
\end{figure}

Our stochastic CC-MDVSP considers all the aforementioned characteristics of the MDVSP, which we call the \emph{planning} portion of the problem. In addition, we consider a set of \emph{service (reliability) requirements} that need to be fulfilled in a risk-averse setting (i.e., achieve the requirements most of the days), which correspond to the \emph{operational} portion of the problem. Motivated by practical considerations (e.g., from \cite{TTC}), we consider an on-time performance (OTP) metric for the start time of each timetabled trip as well as  fairness metrics for related timetabled trips:
\begin{itemize}
    \item \underline{\textit{OTP.}} We say that trip $i\in \trips$ \emph{starts on time} if the execution start time is inside the time window $[\stime_i - \lb, \stime_i + \ub]$, with given $\lb,\ub\geq 0$. Then, we say that a vehicle schedule \emph{fulfills the OTP requirements} if at least $\numserviceTrip =\lfloor\numtrips\cdot\serviceTrips\rfloor$ trips start on time, where $\serviceTrips \in (0,1]$ represents the minimum proportion of the timetabled trips that have to start on time. We note that due to practical considerations, any trip $i$ is not allowed to start earlier than $\stime_i - \lb$, even if the assigned bus arrives earlier; however, it can start later than $\stime_i + \ub$, which would be then identified as \emph{delayed}. On the other hand, we assume that the very first trip of each bus (i.e., the one scheduled right after leaving a depot) starts as early as possible, namely at $\stime_i - \lb$ for such a trip $i \in \trips$.

    \item \underline{\textit{Fairness.}} These conditions ensure that timetabled trips associated with different routes (i.e., a set of trips with common start and end locations) have similar OTP. Specifically, we consider a set of $\numroute$ routes $\routes=\{1,...,\numroute\}$, where a route $r \in \routes$ corresponds to a subset of trips $\trips(r) \subseteq \trips$ with a common start and end location. We note that every trip is part of a single route, thus, sets $\trips(r)$ for $r\in \routes$ constitute a partition of $\trips$. Then, the fairness requirements of a route $r \in \routes$ impose that a proportion of the trips in $\trips(r)$ start on time, that is, at least $\numserviceRoute_r = \lfloor \numtrips^r \cdot \serviceRoute \rfloor$ trips start on time for $I^r= |\trips(r)|$ and a given $\serviceRoute\in (0,1]$. If this condition is satisfied for each route, then we say that a vehicle schedule \emph{fulfills the fairness requirements}.
\end{itemize}

We model these service requirements using a joint chance constraint to represent the risk-averse behavior of transit agencies, that is, fulfill the service requirements for most business days. In particular, we impose that all service requirements must be achieved with a probability of at least $1-\epsilon$, where  $\epsilon \in (0,1)$ is the risk tolerance or, more specifically, the probability threshold of not fulfilling at least one service condition (i.e., violating one of the OTP and fairness requirements).

Similarly to \cite{ricard2022increasing}, we consider a simple model that incorporates these service requirements, which only considers the direct and indirect trips' delays (i.e., delay propagation). In addition, our CC-MDVSP variant considers a recourse action in which bus drivers can reduce their travel times to avoid violating the service requirements. For example, in many public transit systems, drivers might slightly increase the speed to reduce travel time, a practice commonly known as \textit{expressing} \citep{eberlein1997real}. To model this action, we consider $\express_i$ as the maximum units of time reduction that can be made for a timetabled trip $i\in \trips$ to help start subsequent trips on time and, in turn, achieve the OTP.

\begin{example} \label{exa:mdvsp_instance_cc}
We now extend Example \ref{exa:mdvsp_instance} to illustrate the aforementioned metrics for the CC-MDVSP. We consider that trips with inverse directions correspond to the same route (e.g., trips 1 and 2 are on the same route); thus, this instance has 4 routes with 2 timetable trips each. For simplicity, the stochastic setting has two scenarios with equal probability and $\epsilon = 0.5$, that is, at least one scenario has to fulfill the service requirements to satisfy the CC. We set parameters $\lb=\ub=0$, $\express_i = 0$ for all $i \in \trips$, $\serviceTrips = 0.85$ and $\serviceRoute = 0.5$. Thus, a vehicle schedule in a specific scenario fulfills the service requirements if at least $\numserviceTrip =7$ trips start on time and each route has at most $\numserviceRoute = 1$ delayed trip. 

Figure \ref{fig:example-grid} illustrates the traveling times of each scenario using different color schemes: scenario 1 in yellow and scenario 2 in purple. Both scenarios consider average traveling times for trips traversing plane grid edges and an increase of 0.5 units of time if they cross a colored edge associated with each scenario. For example, in scenario 1, trips 1 and 2 have both an increase of 0.5 units of times (i.e., a duration of 5.5 each), but in scenario 2 their duration remains the same.  

The left schedule of Figure \ref{fig:example-grid} violates CC. In particular, trips 3 and 4 are delayed in the first scenario, thus, violating the OTP and fairness constraint. Also, trips 2 and 4 are delayed in the second scenario, which violates the OTP constraint. In contrast, the right schedule is a feasible schedule for the CC-MDVSP since the schedule satisfies the service requirement for both scenarios. Specifically, only one trip is delayed in each scenario (i.e., trip 6 in scenario 1 and trip 4 in scenario 2), which satisfies both the OTP and fairness constraints. \hfill $\square$
\end{example}

\subsection{Multi-commodity Flow Formulation}
We now present a model of the CC-MDVSP. In what follows, we use lowercase bold letters for vectors of decision variables and sub-indices to refer to each variable (e.g., $\bx=(x_1,x_2,x_3)^\top \in \R^3$ is a vector of continuous decision variables).

As discussed in Section \ref{sec:literature}, several deterministic MDVSP models can be considered for an extension to our chance-constrained setting. Most of these formulations rely on a time-space network flow model (e.g., \citep{kliewer2006time} and \citep{kulkarni2018new}) that employ trips' durations to define nodes in the network. Since our problem considers stochastic travel times, this would require building a different time-space network for each possible realization of $\xi$. Alternatively, the MCF formulation \citep{desrosiers1995time} relies on a network that relates timetabled trips that can be sequentially scheduled to the same bus and does not explicitly consider the travel times. 
Therefore, we propose an  MCF formulation for the CC-MDVSP. This formulation represents the problem with a network $\graph=(\nodes,\arcs)$, where $\nodes$ is the set of nodes and $\arcs$ is the set of arcs. We consider one node for each timetable trip $i\in \trips$ and two nodes for each depot $k \in \depots$, that is, $\nodes = \trips \cup \depotsStart \cup \depotsEnd$, where $\depotsStart$ and $\depotsEnd$ are copies of $\depots$ that represent the start and end depots of a vehicle schedule, respectively. 

The set of arcs is given by the compatibility set of timetabled trips $\compatible$, that is, the set of trip pairs that can be scheduled in the same bus. In the deterministic MDVSP, we say that two trips $i,j \in \trips$ are compatible if the scheduled end time of trip $i$ plus the travel time of the deadhead trip from $i$ to $j$ is less than or equal to the scheduled start time of trip $j$. Thus, the compatibility set is
$ \compatible^{\texttt{det}} = \{ (i,j) \in \trips\times \trips: \; \stime_i + \dur_i + \travel_{ij} \leq \stime_j  \}$.
However, travel times in the CC-MDVSP are random variables, so we consider a general representation of $\compatible$ using $\approxdur_i$ and $\approxtravel_{ij}$ as estimates of the travel times for $i,j \in \trips$. For example, a conservative approach sets $\approxdur_i =0$ or $\approxdur_i =\min\{\dur_i(\xi)\}$ (where the minimum is taken over the support of $\xi$), while an average approach considers $\approxdur_i = \Exp[\dur_i(\xi)]$, as in the deterministic MDVSP literature; similarly treating the $\approxtravel$ parameters. Thus, the \emph{planning compatibility set} for the CC-MDVSP is:
\[ 
\compatible = \{ (i,j) \in \trips\times \trips: \; \stime_i + \approxdur_i + \approxtravel_{ij} \leq \stime_j  \}. \]
whose member pairs we refer to as \emph{planning compatible}. 
Then, the set of arcs in the network is given by $\arcs = \compatible \cup \{(k,i): i \in \trips, k\in \depotsStart\} \cup \{(i,k): i \in \trips, k \in \depotsEnd\}$, that is, there is a directed arc: 
\begin{enumerate*}[label=(\roman*)]
    \item for each pair of compatible trips,
    \item from each start depot to every trip, and
    \item from each trip to every end depot.
\end{enumerate*}
Note that set $\arcs$ has cardinality  $\numarcs = |\compatible| + 2KI$, and it considers that all trips can be directly linked to a depot.

\begin{example}
Figure \ref{fig:example-network} shows the network $\graph=(\nodes,\arcs)$ for Example \ref{exa:mdvsp_instance_cc} using the planning compatibility set $\compatible$ with average times. Gray nodes represent the depot copies (i.e.,  with $k_1^s, k_2^s \in \depotsStart$ and $k_1^e, k_2^e \in \depotsEnd$) and black nodes correspond to timetable trips. Solid arrows link node trips that are in the planning compatibility set.  Dashed arrows represent pull-out and pull-in arcs \citep{kliewer2006time}, that is, deadhead trips from and to the depots, respectively. We represent a subset of all pull-in and pull-out arcs in the network to avoid overcrowding the graph. The network also illustrates the schedule in the left grid of Figure \ref{fig:example-grid} with shaded arcs. 
\hfill $\square$
\end{example}

\begin{figure}[htb]
    \centering
    \begin{tikzpicture}[->,>=stealth',shorten >=1pt,auto,node distance=1cm]

\node[node depot] (k1s) at (0,1) {$k_1^s$};
\node[node depot] (k2s) at (0,-1) {$k_2^s$};

\node[node trip] (n1) at (2,1) {$1$};
\node[node trip] (n8) at (2,-1) {$8$};

\node[node trip] (n3) at (3.5,0.7) {$3$};
\node[node trip] (n6) at (3.5,-0.7) {$6$};

\node[node trip] (n4) at (5.2,1.5) {$4$};
\node[node trip] (n5) at (5.2,-1.5) {$5$};

\node[node trip] (n2) at (7,1.1) {$2$};
\node[node trip] (n7) at (7,-1.1) {$7$};

\node[node depot] (k1e) at (9,1) {$k_1^e$};
\node[node depot] (k2e) at (9,-1) {$k_2^e$};

\path[every node/.style={font=\sffamily\small}]
(k1s) 
edge[pull arc sol] node [left] {} (n1)
edge[pull arc] node [left] {} (n8)
(k2s)
edge[pull arc] node [left] {} (n1)
edge[pull arc sol] node [left] {} (n8)
(n1)
edge[arc sol] node [left] {} (n3)
edge[arc] node [left] {} (n5)
edge[arc] node [left] {} (n4)
edge[arc] node [left] {} (n2)
edge[arc] node [left] {} (n7)
(n8)
edge[arc] node [left] {} (n3)
edge[arc sol] node [left] {} (n6)
edge[arc] node [left] {} (n5)
edge[arc] node [left] {} (n4)
edge[arc] node [left] {} (n2)
edge[arc] node [left] {} (n7)
(n3)
edge[arc sol] node [left] {} (n4)
edge[arc] node [left] {} (n2)
edge[arc] node [left] {} (n7)
(n6)
edge[arc sol] node [left] {} (n5)
edge[arc] node [left] {} (n2)
edge[arc] node [left] {} (n7)
(n5)
edge[arc] node [left] {} (n2)
edge[arc sol] node [left] {} (n7)
(n4)
edge[arc sol] node [left] {} (n2)
edge[arc] node [left] {} (n7)
(n2)
edge[pull arc sol] node [left] {} (k1e)
edge[pull arc] node [left] {} (k2e)
(n7)
edge[pull arc] node [left] {} (k1e)
edge[pull arc sol] node [left] {} (k2e)
;
\end{tikzpicture}
    \caption{Network representation for the CC-MDVSP instance described in Examples \ref{exa:mdvsp_instance} and \ref{exa:mdvsp_instance_cc}. }
    \label{fig:example-network}
\end{figure}
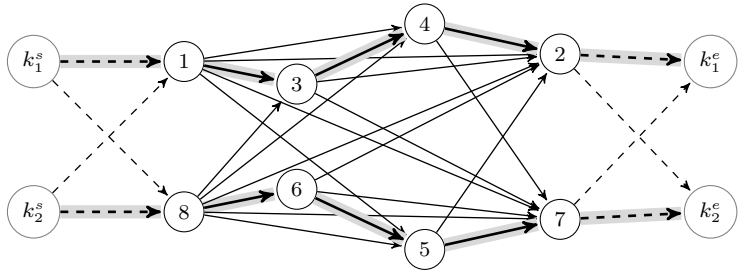

We note that the planning compatibility set $\compatible$ only defines the set of allowed consecutive trip pairs based on the travel time estimates, whereas the proposed chance-constrained model decides which trip sequences are indeed feasible based on the travel time realizations (i.e., the operational portion of the model). As previously mentioned, this set $\compatible$ can always be built in a conservative manner so that no feasible schedule is cut off. However, building $\compatible$ via a less conservative but ``safe'' approach (i.e., not removing any feasible schedule) would reduce the number of arcs in the network and, as such, help solve the optimization model more efficiently. Also, this definition of $\compatible$ allows us to easily incorporate any other compatibility considerations/preferences.

The MCF model creates timetabled trip sequences (i.e., one sequence per bus) and assigns them to depots so that all trips are covered exactly once. The main decision vector is $\bx \in \B^{\numarcs \times \numdepots}$, where $x_{ijk}\in \{0,1\}$ represents visiting node $i \in \nodes$ right before node $j \in \nodes$, with a bus associated to depot $k \in \depots$. Then, our MCF model for the CC-MDVSP  is given by:
\begin{align}
    \min \; & \sum_{k \in \depots}\sum_{(i,j)\in \arcs}\cost_{ijk} x_{ijk}  \tag{\textsf{CC-MCF}} \label{model:chance} \\
    \mbox{s.t.}\;
    & \bx \in \mcfset  \nonumber \\ 
    & \Prob\{\bx \in \ccset(\xi)\} \geq 1 - \ccprob. \label{cc:chanceConstraint}
\end{align}

The objective function of \eqref{model:chance} represents the travel costs, where, for any given $k \in \depots$, $\cost_{ijk} = \cost_{ij}$ for $(i,j)\in \compatible$,  $\cost_{kik} = \cost_{ki}$, and $\cost_{ikk} = \cost_{ik}$ for $i \in \trips$. Set $\mcfset$ corresponds to the planning constraints representing the candidate vehicle schedules. This set is described by the deterministic MCF constraints over the network $\graph=(\nodes,\arcs)$ as
\begin{subequations}
\begin{align}
    \mcfset = \Big\{ & \bx \in \B^{A\times K}:  \nonumber \\
    & \sum_{k \in \depots}
    \sum_{i : (i,j) \in \arcs}
    x_{ijk} = 1, & \forall j \in \trips, \ \  \label{cc:tripOnce} \\
    & \sum_{i \in \trips} x_{kik}  \leq \capacity_k, & \forall k \in \depots, \ \  \label{cc:capacity} \\
    & 
    \sum_{j : (j,i) \in \arcs}
    x_{jik} - 
    \sum_{j : (i,j) \in \arcs}
    x_{ijk} = 0, & \forall i \in \trips, \; k \in \depots, \ \  \label{cc:flowBalance} \\
    & x_{k'ik}= x_{ik'k} = 0  & \forall k,k'\in \depots,\; k \neq k',\; i \in \trips \label{cc:consistenDepot} \Big\}.
\end{align}
\end{subequations}

Constraints \eqref{cc:tripOnce} ensure that each trip is scheduled exactly once. Constraints \eqref{cc:capacity} represent the capacity of the depots, that is, the number of buses (i.e., timetabled trip sequences) that can be assigned to a given depot. Constraints \eqref{cc:flowBalance} correspond to flow balance equations and constraints \eqref{cc:consistenDepot} ensure that the start and end depot nodes coincide with the associated depot. 

Lastly, constraint \eqref{cc:chanceConstraint} corresponds to the operation portion of the problem, that is, the joint chance constraint that models the service reliability. We consider three additional decision variables to model the service reliability requirements set $\ccset(\xi)$. For each trip $i \in \trips$, variable $v_i\in \{0,1\}$ indicates if trip $i$ starts on time,  $y_i\geq 0$ represents the start time of the trip, and $u_i \in [0,\express_i]$ corresponds to the amount of time used by the bus to decrease the duration of the trip. Then, the set of operational constraints is
\begin{subequations}
\begin{align}
    \ccset(\xi) = \Big\{ \bx &\in \{0,1\}^{\numarcs\times \numdepots} : \;  \exists \; (\bv, \by, \bu) \mbox{ such that} \nonumber \\
    & \sum_{i \in \trips} v_i  \geq \numserviceTrip, \label{cc:service-trips}\\
    & \sum_{i \in \trips(r)} v_i  \geq \numserviceRoute_r, & \forall r \in \routes, \ \  \label{cc:service-routes}\\
    &  y_j  + \dur_j(\xi) + \travel_{ji}(\xi) - u_j -\bigMstart_{ji}\left( 1 - \sum_{k \in \depots} x_{jik} \right) \leq y_i, & \forall (j,i)\in \compatible, \ \  \label{cc:start-trip}\\
    & \travel_{ki}(\xi) -\bigMstart_{ki} \left( 1 - x_{kik} \right) \leq y_i, & \forall i \in \trips, k \in \depots, \ \  \label{cc:start-depot} \\
    & y_i  \leq  \stime_i + \ub \cdot v_i  + \bigMotp_i(1 - v_i ), & \forall i \in \trips, \ \  \label{cc:start-ub}\\
    & \stime_i - \lb \leq y _{i}, & \forall i \in \trips, \ \ \label{cc:start-lb}  \\
    & v _{i} \in \B, \; y_i \geq 0,\; u_i \in [0,\express_i ] & \forall i \in \trips\Big\}. \label{cc:domains}
\end{align}
\end{subequations}

Constraints \eqref{cc:service-trips} and \eqref{cc:service-routes} enforce the service requirements, that is, the minimum number of trips that start on time across all trips and for each route, respectively. Constraints \eqref{cc:start-trip} and \eqref{cc:start-depot} model the start time of each trip given the set of scheduled decisions. Constraints \eqref{cc:start-ub}  relate the start time of each trip with their respective on-time decision variables, and constraints \eqref{cc:start-lb} ensure that start times are not earlier than what is mandated. Note that parameters $\bigMstart_{ji},\bigMstart_{ki}, \bigMotp_i >0$ for $i,j\in \trips$ and $k\in \depots$ are sufficiently large values chosen to model the logical implications of constraints \eqref{cc:start-trip}-\eqref{cc:start-ub} properly, respectively. We refer the reader to Section \ref{sec:lp-cuts} for a discussion on tight big-M values for this formulation. 
Lastly, we assume that all trips scheduled first on a sequence (i.e., right after leaving the depot) always start on time, which makes 
\eqref{cc:start-depot} redundant, but we leave them in the model for completeness. The validity of this assumption arises in real transit agencies that can adjust their drivers' schedules to guarantee that they always start the first trip on time.

\section{Branch-and-Cut Decomposition Scheme} \label{sec:decomposition}

This section introduces the methodology employed to address the \eqref{model:chance} model. We use the commonly employed SAA approach that transforms the CCP model into a large-scale MILP model \citep{luedtke2008sample} to obtain near-optimal solutions to the CC-MDVSP. Given the size of the model and its well-known weak LP relaxation \citep{luedtke2010integer}, we propose a B\&C approach to handle these issues due to its success in other applications with continuous recourse variables \citep{luedtke2014branch}. 

Due to the existence of binary recourse variables, there is no readily suitable B\&C approach for our problem. Therefore, we devise problem-specific cut-generation procedures to solve the resulting SAA model. Before describing our cut families in Section \ref{sec:cutgeneation},  this section provides the SAA formulation and an overview of the decomposition framework. Moreover, we show that our (scenario) subproblems can be optimally solved via a polynomial-time greedy algorithm, which is key when designing some of our cut-generation variants. We also devise valid inequalities for each scenario realization to strengthen the master problem relaxation.

\subsection{SAA Formulation}
The SAA formulation for \eqref{model:chance} considers a set of $\numscenarios$ scenarios $\scenarios$ obtained by sampling $\xi$, $\{ \xi^\s\}_{s \in \scenarios}$, where each scenario $\s \in \scenarios$ has probability $1 /{\numscenarios}$. In what follows, we use 
superscript notation to represent the specific value of a stochastic parameter omitting the underlying $\xi^\s$ parametrization for brevity (e.g., $\dur_i^\s=\dur_i(\xi^s)$ represents the duration of trip $i \in \trips$ in scenario $\s\in \scenarios$).

For the SAA model, we introduce a new set of binary variables $\bz \in \B^{\numscenarios}$, where $z_\s = 0$ if the service requirements are \textit{met} for scenario $\s \in \scenarios$ (i.e., $\bx \in \ccset(\xi^\s)$) and $z_\s = 1$ if the requirements can be violated. Then, the scenario-based formulation 
is given by: 
\begin{subequations}
\begin{align}
    \min \; & \sum_{k \in \depots}\sum_{(i,j)\in \arcs}\cost_{ijk} x_{ijk}  \tag{\textsf{SAA-MCF}} \label{model:chance-saa} \\
    \mbox{s.t.}\;
    & \bx \in \mcfset \nonumber \\
    & \sum_{\s \in \scenarios} z_\s \leq  \lfloor \numscenarios \epsilon \rfloor , \label{ccs:probability}\\
    & (\bx, z_\s) \in \subproblem^\s, & \forall \s \in \scenarios, \label{ccs:scenario}\\
    & z_s \in \B & \forall \s \in \scenarios.  \nonumber
\end{align}
\end{subequations}

Model \eqref{model:chance-saa} maintains the objective function and the set of planning constraints of \eqref{model:chance}, but replaces the CC \eqref{cc:chanceConstraint} (i.e., the operational portion of the problem) with a set of linear constraints for each scenario. Specifically, constraint \eqref{ccs:probability} enforces the maximum number of scenarios that can violate the service requirements. Set $\subproblem^\s$ corresponds to the service requirements $\ccset(\xi)$ for each scenario $\s\in \scenarios$, which is given by
\begin{subequations}
\begin{align}
    \subproblem^\s = \Big\{ & (\bx,z) \in \B^{\numarcs\times \numdepots}\times \B: \;  \exists \; (\bv, \by, \bu) \mbox{ such that} \nonumber \\
    & \sum_{i \in \trips} v_i  \geq \numserviceTrip (1 -  z), \label{ccs:service-trips}\\
    & \sum_{i \in \trips(r)} v_i  \geq \numserviceRoute_r (1 -  z), & \forall r \in \routes, \ \  \label{ccs:service-routes}\\
    &  y_j  + \dur_j^\s + \travel_{ji}^\s - u_j  -\bigMstart_{ji}\left( 1 - \sum_{k \in \depots} x_{jik} \right) \leq y_i, & \forall (j,i)\in \compatible, \ \  \label{ccs:start-trip}\\
    & \travel_{ki}^\s -\bigMstart_{ki} \left( 1 - x_{kik} \right) \leq y_i, & \forall i \in \trips, k \in \depots, \ \  \label{ccs:start-depot} \\
    & y_i  \leq  \stime_i + \ub v_i  + \bigMotp_i(1 - v_i ), & \forall i \in \trips, \ \  \label{ccs:start-ub}\\
    & \stime_i - \lb \leq y _{i}, & \forall i \in \trips, \ \  \label{ccs:start-lb}  \\
    & v _{i} \in \B, \; y_i \geq 0,\; u_i \in [0,\express_i] & \forall i \in \trips\Big\}. \label{ccs:domains}
\end{align}
\end{subequations}
\noindent The difference between $\subproblem^\s$ and $\ccset(\xi)$ is the inclusion of the scenario variable $z\in\B$ and the realization of the random variables for scenario $\s\in \scenarios$. In particular, the OTP and fairness constraints \eqref{ccs:service-trips} and \eqref{ccs:service-routes} are active if the service requirements for a particular scenario can be satisfied (i.e., $z=0$). The remaining constraints \eqref{ccs:start-trip}-\eqref{ccs:domains} are analogous to \eqref{cc:start-trip}-\eqref{cc:domains} for a scenario realization.

It is well-known that an SAA model gives good approximations for its corresponding CCP model for a limited number of scenarios, and its optimal value converges to the CCP optimal value as the number of scenarios grows \citep{luedtke2008sample}. Thus, given a sufficiently large number of scenarios, \eqref{model:chance-saa} is guaranteed to yield a near-optimal solution for \eqref{model:chance}. 

\subsection{Decomposition Framework}

Model \eqref{model:chance-saa} is a large-scale problem due to scenario copies of variables $z$, $\bv$, $\by$, $\bu$, and their corresponding set of constraints $\subproblem^\s$ for each $\s\in \scenarios$. We propose a B\&C framework that divides the problem into a master problem for all planning constraints and one subproblem for each scenario representing the operational constraints and variables. 

Our master problem considers all the vehicle scheduling decisions $\bx\in \B^{\numarcs\times \numdepots}$ and the scenario variables $\bz\in \B^\numscenarios$ indicating whether or not the service requirements are met for each scenario. Thus, the master problem is given by
\begin{subequations}
\begin{align}
    \min \; & \sum_{k \in \depots}\sum_{(i,j)\in \arcs}\cost_{ijk} x_{ijk}  \tag{\textsf{SAA-Master}} \label{model:master} \\
    \mbox{s.t.}\;
    & \bx \in \mcfset, \nonumber \\
    & \sum_{\s \in \scenarios} z_\s \leq \lfloor \numscenarios \epsilon \rfloor, \nonumber \\
    & \sum_{k \in \depots}\sum_{(i,j)\in \arcs}\theta_{ijk} x_{ijk} \leq \theta_0 z_s + \theta_1 (1-z_\s) , & \forall (\btheta, \theta_0, \theta_1) \in \validset^\s, \; \s \in \scenarios, \label{master:valid-ine}\\
    & \sum_{(i,j)\in \arcs}\lambda_{ij}\sum_{k \in \depots} x_{ijk} \leq \lambda_0 - 1  + z_\s ,  & \forall (\blambda, \lambda_0) \in \cutset^\s, \; \s \in \scenarios,  \label{master:cuts-added}\\
    & z_\s \in \B ,& \forall \s \in \scenarios. \nonumber
\end{align}
\end{subequations}
Note that \eqref{model:master} includes all the MCF planning constraints $\mcfset$, which are common for all scenarios, and inequality \eqref{ccs:probability} that links all the scenario variables. The model also includes a set of valid inequalities \eqref{master:valid-ine} that relate the scheduling decisions with each scenario variable, which we formally present in Section \ref{sec:master-valid-inequalities}. Lastly, \eqref{master:cuts-added} corresponds to the set of constraints added during the B\&C algorithm, which are discussed in Section \ref{sec:cutgeneation}. Note that $\validset^\s$ and $\cutset^\s$ are the set of coefficients\footnote{These sets also determine the constants in the constraints but are referred to as coefficients for conciseness.} for constraints  \eqref{master:valid-ine} and \eqref{master:cuts-added} for each scenario $\s\in \scenarios$, respectively. 

The master problem considers all the variables and constraints linking different scenarios, which allows us to create one subproblem for each scenario. The primary goal of solving a subproblem is to check if a candidate vehicle schedule $\bxhat \in \mcfset$ meets the service requirements for a specific scenario. To do so, we consider subproblem \eqref{model:subproblem} for each scenario $\s\in \scenarios$ and a fixed  master solution $\bxhat$:
\begin{equation}
\zbar \in \argmin_{z\in \B} \left\{ z  :\; (\bxhat,z) \in \subproblem^\s
\right\}. \tag{\textsf{SAA-Sub($\s, \bxhat$)}} \label{model:subproblem} 
\end{equation}
The model includes all the variables and constraints in $\subproblem^\s$ and minimizes the scenario variable $z$ to indicate if $\bxhat$ meets the service requirements for $\s\in \scenarios$ when the optimal value is zero. 

\begin{figure}[tb]
    \centering
    \includegraphics[width=\textwidth]{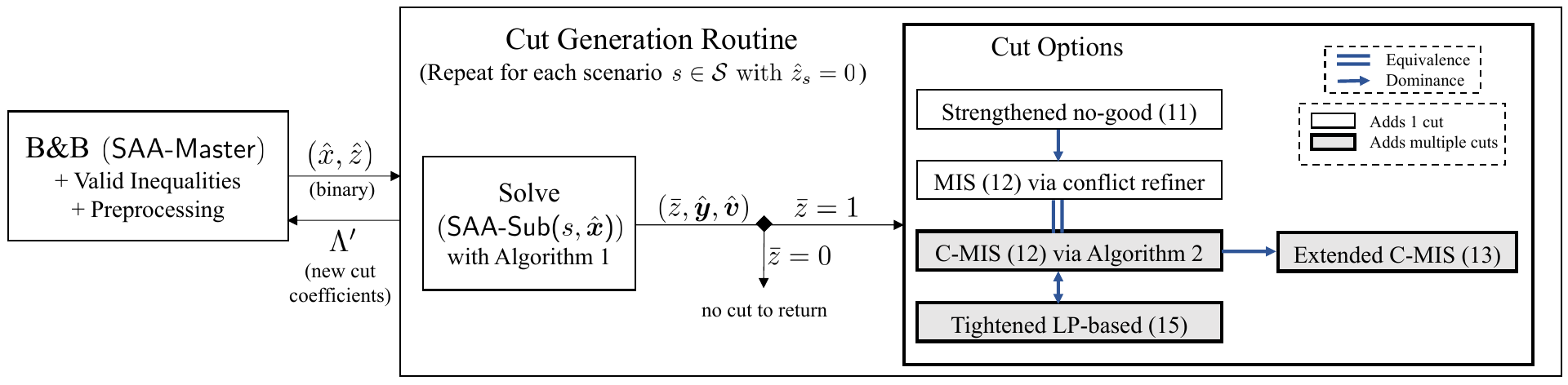}
    \caption{Diagram of the decomposition scheme and the relationship between the proposed cuts.}
    \label{fig:decomposition-diagram}
\end{figure}

Figure \ref{fig:decomposition-diagram} illustrates the main components of our B\&C procedure. We perform a branch-and-bound (B\&B) search over \eqref{model:master} until we find a binary feasible solution $(\bxhat, \bzhat)$. This solution is passed to our cut generation routine. This routine iterates over all the scenarios where the master problem indicates that the vehicle schedule meets the service requirements (i.e., $\zhat_s = 0$) and checks if that is the case by solving \eqref{model:subproblem} using the greedy algorithm (i.e., Algorithm \ref{alg:greedy} described in Section \ref{sec:greedy}). If the optimal value of the subproblem $\zbar$ is one (i.e., at least one of the requirements cannot be met), then we generate one or more cuts for that scenario and add the cut coefficients to the corresponding set ${\cutset^\s}'$. Our procedure iterates over all scenarios searching for as many cuts as possible, but alternatives can be considered (e.g., add at most one cut per iteration). Lastly, we return the new set of cut coefficients $\cutset'$ and add them to the master problem as $\cutset = \cutset \cup \cutset'$.

Figure \ref{fig:decomposition-diagram} also presents several components of our procedure described in the following sections. For instance, the master problem can consider preprocessing steps commonly used in the determinist MDVSP literature (e.g., variable fixing and odd-cycle cuts \citep{hadjar2006branch,groiez2013separating}) and our valid inequalities detailed in Section \ref{sec:master-valid-inequalities}. Lastly, it summarizes our proposed cut-generation alternatives and their relationship, which are presented in Section \ref{sec:cutgeneation}.

\subsection{Subproblem Greedy Algorithm} \label{sec:greedy}

Solving subproblem \eqref{model:subproblem} for each candidate vehicle schedule $\bxhat$ and scenario $\s\in \scenarios$ is one of the main components of our cut generation routine. Any MILP solver can directly solve this subproblem, but it could be computationally expensive as the number of timetabled trips grows. In what follows, we present a polynomial-time greedy algorithm that returns an optimal solution to \eqref{model:subproblem}. This greedy procedure is used to solve \eqref{model:subproblem} as illustrated in Figure \ref{fig:decomposition-diagram}. Moreover, this procedure is a crucial component to some of our proposed cut-generation routines, as explained in Section \ref{sec:cutgeneation}.

To explain the greedy algorithm and other components of our methodology, we represent the schedule $\bxhat$ as a set of buses $\bus(\bxhat) = \{ \bus_1,..., \bus_\numbuses\}$ (i.e., trips sequences), where $\numbuses$ corresponds to the total number of buses in the schedule. Specifically, $\bus(\bxhat)$ is a partition of $\trips$ such that each $\bus_b$ with $b \in \{1,...,\numbuses\}$ contains all the trips in the same trip sequence given by $\bxhat$. For each bus $b$, $\bus_b$ is an ordered set with respect to the trip's schedule $\bxhat$ (e.g., the first trip in $\bus_b$ is the first trip in the schedule of bus $b$). For example, the schedule in the left grid of Figure \ref{fig:example-grid} contains two buses given by $\bus_1=(1,3,4,2)$ and $\bus_2=(8,6,5,7)$.

\begin{algorithm}[htb]
\small
    \linespread{1.2}\selectfont
	\caption{Greedy Algorithm for \eqref{model:subproblem} } \label{alg:greedy}
	\begin{algorithmic}[1]
		\Procedure{\algGreedy }{$\s$, $\bus(\bxhat)$}
		\State Initialize $u^*_i = \express_i$ for all $i \in \trips$
		\For {$b \in \{1,...,\numbuses\}$}
            \For {$p \in \{1,..., |\bus_b| \}$}
            \State Let $i$ be the index of the $p$-th trip in $\bus_b$
            \If { $p = 1$ (i.e., $i$ is the first trip in $\bus_b$)}  $y^*_i =\stime_i - \lb$ and $v_i^* = 1$
		\Else {} 
		\begin{equation}
		    y^*_i = 
          \max\left\{ \stime_i - \lb, \;  y^*_j + \dur^\s_j + \travel^\s_{ji} - u_j^*  \right\}
		\label{eq:earliest-start-time}
		\end{equation}
            \State Set $v^*_i = 1$ if $y_i^* \in [\stime_i -\lb, \stime_i +\ub]$ and $v^*_i = 0$ otherwise
        \EndIf
            \State Set $j$ as the last trip iterated (i.e., $j = i$)
		\EndFor
            \EndFor
		\If { $\sum_{i \in \trips(r)} v_i^* \geq \numserviceRoute_r$ for all $r \in \routes$ and $\sum_{i \in \trips} v_i^* \geq \numserviceTrip$} $z^* = 0$
		\Else {} $z^* = 1$
		\EndIf
		\State \algReturn\ $(z^*,\by^*,\bv^*, \bu^*)$
		\EndProcedure
	\end{algorithmic} 
\end{algorithm}

Algorithm \ref{alg:greedy} shows our greedy procedure that finds the earliest start time of each timetabled trip for a given vehicle schedule $\bxhat$. It first fixes the expressing variables to their maximum value to consider the minimum total duration of each trip (line 2). We then iterate over the trips of each bus following their schedule order (lines 3-5) to guarantee that all predecessor trips' start times are computed before a given trip. If a trip is the first in the bus, it starts as early as possible and is always on time (line 6). Otherwise, \eqref{eq:earliest-start-time} computes the earliest start time considering the previous trip in $\bus_b$ and the scheduled start time. These times are then used to determine if a trip arrives on time (line 8). Lastly, we check if the OTP and fairness constraints are met, set the value of $z^*$ accordingly, and return the solution (lines 9-10). This algorithm has a computational complexity of $O(\numtrips)$  because we iterate over each trip only once to compute the earliest start times $\by^*$ and the on-time variables $\bv^*$.  As stated in Proposition \ref{prop:greedy}, whose proof is given in Appendix \ref{appendix:prop:greedy-proof}, this algorithm returns an optimal solution for \eqref{model:subproblem} for a given $\bxhat$ and $\s\in \scenarios$.

\begin{proposition} \label{prop:greedy}
For a master solution $\bxhat$ and scenario $\s\in \scenarios$,  Algorithm \ref{alg:greedy} returns an optimal solution for \eqref{model:subproblem}, where $\by^*$ is the earliest start times of each timetabled trip.
\end{proposition}

\subsection{Master Problem Valid Inequalities} \label{sec:master-valid-inequalities}

As previously mentioned, we consider a set of valid inequalities \eqref{master:valid-ine} to strengthen the master problem \eqref{model:master} by relating the scheduling variables $\bx$ with the scenario variables $z_\s$ for each $\s \in \scenarios$. We now describe these inequalities and show their validity. 

The general idea is to use the scenario realizations to determine the set of trip pairs that will lead to delays. We first compute the operational compatibility set for each scenario $\s \in \scenarios$ as
$ \compatible^\s = \{ (i,j) \in \trips\times \trips: \; \stime_i - \lb + \dur_i^\s + \travel_{ij}^\s -\express_i \leq \stime_j + \ub  \}. $
This definition states that if we schedule trips $(i,j)\notin \compatible^\s$ on the same bus, trip $j$ will be delayed in scenario $\s \in \scenarios$. We use this observation to derive constraints that relate trip pairs that lead to known delays for any given scenario:
\begin{equation} \label{eq:valid-trip}
    \sum_{j \in \trips} \sum_{i: (i,j)\in \compatible\setminus\compatible^\s} \sum_{k \in \depots}x_{ijk} \leq z_\s \numtrips_\s + (1-z_\s)\numserviceTrip, \qquad \forall \s\in \scenarios.
\end{equation}
where $\numtrips_\s = |\{ j \in \trips:\; \exists i\in \trips \text{ such that }  (i,j)\in \compatible\setminus\compatible^\s \}|$ is the maximum number of trips that can be delayed in $\s$ because one of its possible predecessors.  Note that set $\compatible\setminus\compatible^\s$ corresponds to all trip pairs that are feasible in the planning phase but lead to delays in scenario  $s\in \scenarios$. Thus, for each scenario $\s\in \scenarios$, inequality \eqref{eq:valid-trip} considers all trips $j$ that have a predecessor in $\compatible$ that will lead to a delay in scenario $\s$.  Since the depot does not affect the compatibility set and each trip can only be associated with one depot, we can consider all possible depot assignments. Thus, the left-hand side of \eqref{eq:valid-trip} counts the number of delayed trips, and the right-hand side enforces $z_\s=1$ if the number of delayed trips surpasses the OTP requirements.

Similarly, we create inequalities that represent the fairness requirements for each route:
\begin{equation} \label{eq:valid-route}
    \sum_{j \in \trips(r)} \sum_{i: (i,j)\in \compatible\setminus\compatible^\s} \sum_{k \in \depots}x_{ijk} \leq z_\s I^r_\s + (1-z_\s) \numserviceRoute_r, \qquad \forall \s\in \scenarios, r \in \routes,
\end{equation}
where $\numtrips^r_\s = |\{ j \in \trips(r):\; \exists i\in \trips \text{ such that }  (i,j)\in \compatible\setminus\compatible^\s \}|$ has an analogous mean to $\numtrips_\s$. 

The main difference between \eqref{eq:valid-trip} and \eqref{eq:valid-route} is that the latter considers delays of trips associated with a specific route. Proposition \ref{prop:valid-ine} states the validity of the inequalities and their coefficient values considering the general form of \eqref{master:valid-ine}. The proof follows from the construction of the compatibility sets, the OTP definitions for each scenario, and that each trip must be visited exactly once (i.e., constraint \eqref{cc:tripOnce}).

\begin{proposition} \label{prop:valid-ine}
Inequalities \eqref{eq:valid-trip} are valid for \eqref{model:master} with coefficients:
\[
\theta_{ijk} = \begin{cases}
1, & (i,j)\in \compatible\setminus\compatible^\s,\; k \in \depots, \\
0, & \text{otherwise},
\end{cases}
\qquad \theta_0 = \numtrips, \qquad  \theta_1 = \numserviceTrip.
\]
The analogous result is true for \eqref{eq:valid-route}.
\end{proposition}

Lastly, we note that valid inequalities \eqref{eq:valid-trip} and \eqref{eq:valid-route} do not dominate each other because each family of inequalities considers a different set of trips (i.e., either $\trips$ or $\trips(r)$ for some $r\in \routes$). In fact, \eqref{eq:valid-trip} can consider trips for different routes, while \eqref{eq:valid-route} only considers possible delays of trips of a particular route. Therefore, both families of inequalities can be useful to strengthen \eqref{model:master}. 
\section{Cut Generation Procedures} \label{sec:cutgeneation}

We now describe several cut generation procedures for our B\&C decomposition framework. As previously mentioned, subproblem \eqref{model:subproblem} is non-convex, so we cannot employ out-of-the-box Benders cuts as proposed by \cite{luedtke2014branch}. Therefore, we develop our own cut-generation algorithms that leverage the underlying structure of the problem. 

A naive cut generation alternative is to use no-good cuts, which are common in the logic-based Benders Decomposition (LBBD) literature \citep{hooker2003logic}. We can adapt these cuts for our B\&C decomposition to relate a vehicle schedule with the scenario variables representing the fulfillment of the service requirements. In particular, the no-good cut for scenario $\s\in \scenarios$ in which a master solution $\bxhat$  does not fulfill the service requirements is
\begin{equation}\label{eq:no-good}
    \sum_{(i,j,k) \in \scheduleDepotOne} x_{ijk} + \sum_{(i,j,k) \in \scheduleDepotZero} (1 - x_{ijk}) \leq \numscheduleDepot - 1 +  z_\s,
\end{equation}
where sets $\scheduleDepotZero = \{(i,j,k) \in \arcs\times \depots:\; \xhat_{ijk}= 0 \}$ and $\scheduleDepotOne = \{(i,j,k) \in \arcs\times \depots:\; \xhat_{ijk}= 1 \}$ represent the variable indices fixed to zero and one, respectively, for the vehicle schedule $\bxhat$, and  $\numscheduleDepot = |\scheduleDepotOne \cup \scheduleDepotZero|$ is the size of the vehicle schedule assignment.

We can improve the standard no-good cut \eqref{eq:no-good} by considering the structure of \eqref{model:chance-saa}. First, it is sufficient to include the variables associated with the actual sequence assignments (i.e., set $\scheduleDepotOne$) because all trips have to be sequenced exactly once and the remaining variables are known to be zero, as stated in \eqref{cc:tripOnce}. Second, the depot associated with a trip sequence is irrelevant in practice to calculate the OTP and fairness requirements because  the first trip in a bus always starts as early as possible. Thus, we can consider any possible depot assignment for the selected sequences. Then, a stronger no-good cut for $\bxhat$ and  scenario $\s\in \scenarios$ is
\begin{equation}\label{eq:no-good-strong}
     \sum_{(i,j) \in \schedule}\sum_{k \in \depots} x_{ijk} \leq \numschedule - 1 +  z_\s,
\end{equation}
where $\schedule= \{(i,j) \in \compatible :\; \exists k \in \depots \mbox{ such that } \xhat_{ijk}= 1 \}$ corresponds to the set of trip pairs sequenced in solution $\bxhat$  ignoring the depot assignments, and $\numschedule = |\schedule|$ is the number of trips pairs. It is clear that \eqref{eq:no-good-strong} is stronger than  \eqref{eq:no-good} because the latter considers all possible depot assignments to the trip pairs and also uses a smaller right-hand side constant. Nonetheless, cut \eqref{eq:no-good-strong} is still quite weak for our decomposition framework because it considers a  single sequencing assignment (i.e., $\xhat_{ij\cdot}$ solution) from a set of exponentially many options. 

In what follows, we present three different procedures to generate stronger cuts that consider a subset of trip pairs (i.e., partial trip sequences) that lead to vehicle schedules with unmet service requirements. Namely, Sections \ref{sec:mis-conlfic} and \ref{sec:mis-greedy} present infeasible set (IS) cuts with two different approaches to find the associated set: (i) a conflict refinement procedure from commercial solvers to find a minimal IS (MIS) and (ii) an iterative procedure based on the greedy algorithm solution to find what we call a {\it constraint-based MIS} (C-MIS). Section \ref{sec:mis-strong} shows how to strengthen the C-MIS cuts by exploiting the procedure to find C-MISs. Section \ref{sec:lp-cuts} shows how to find cuts using the LP relaxation of the subproblem \eqref{model:subproblem} and the solution found by our greedy algorithm. We also show that there is a strong relationship between the C-MIS cuts and the LP-based cuts.  Lastly, Section \ref{sec:b&c-zcont} shows how we can improve the B\&C procedure given the specific structure of our cuts. Figure \ref{fig:decomposition-diagram} summarizes all the proposed cut-generation alternatives, the relationship between each other, and the number of different cuts that each one can provide.

\subsection{MIS Cuts via Conflict Refinement} 
\label{sec:mis-conlfic}

A common practice in the literature is to strengthen no-good cuts by considering the minimal set of variable assignments that lead to an infeasible solution \citep{rahmaniani2017benders}. Specifically, for a given master solution $\bxhat$ and scenario $\s\in \scenarios$, an IS  is any subset of trip pairs $\scheduleIS \subseteq \schedule$ such that any schedule with this subset of trip pairs fails to fulfill the service requirements (i.e., violate \eqref{cc:chanceConstraint}). An MIS is an IS $\scheduleMIS \subseteq \schedule$ such that no subset of $\scheduleMIS$ is also an IS. Then, an MIS cut is given by
\begin{equation}\label{eq:mis-cut}
    \sum_{(i,j) \in \scheduleMIS}\sum_{k \in \depots} x_{ijk} \leq |\scheduleMIS| - 1 +  z_\s.
\end{equation}

One way to find an MIS is to use the conflict refinement tools of commercial solvers, such as CPLEX or Gurobi. Specifically, for a given $\bxhat$ that does not meet the service requirements in scenario  $\s\in \scenarios$, we modify subproblem \eqref{model:subproblem} by fixing variable $z = 0$ and, thus, making the problem infeasible. We then run the conflict refinement procedure to find the minimum set of constraints associated with $\bxhat$ that leads to a conflict (i.e., constraints \eqref{ccs:start-trip}). Finally, $\scheduleMIS$ is given by the trip pairs associated with the conflicting constraints.

The MIS cuts \eqref{eq:mis-cut} are guaranteed to remove the same or more solutions than \eqref{eq:no-good-strong} because $\scheduleMIS\subseteq \schedule$. One issue of this approach is its computational time because it needs to solve a MILP and run the conflict refinement procedure for each candidate's master problem solution and scenario. Moreover, there might be multiple MISs for a given $\bxhat$ and $\s\in \scenarios$, but the conflict refinement procedure only returns one. 

\subsection{C-MIS Cuts Using Greedy Solution} \label{sec:mis-greedy}

Given the drawbacks of using MILP-based conflict refinement to find an MIS and the lack of alternative efficient procedures, we develop a polynomial-time procedure to find C-MISs using the solution of Algorithm \ref{alg:greedy}. We define a C-MIS as an MIS obtained when we relax \eqref{model:subproblem} by considering a  single violated service requirement constraint (i.e., one inequality of \eqref{ccs:service-trips} and \eqref{ccs:service-routes}). To build a C-MIS, we develop a two-step algorithm that: 
\begin{enumerate*}[label=(\roman*)]
    \item identifies a set of delayed trips that constitute the base to build a C-MIS, and
    \item finds a subsequence of trips in the master candidate schedule that leads to those delays. 
\end{enumerate*}

In what follows, consider a master solution $\bxhat$ that violates the service requirements for a scenario $\s\in \scenarios$ and its corresponding solution 
$(z^*, \by^*, \bv^*, \bu^*)$ from Algorithm \ref{alg:greedy}. We use the notation $\delaytrips = \{ i \in \trips: v^*_i =0 \}$ to represent the set of trips that are delayed in scenario $s$, and the bus notation $\bus(\bxhat) = \{ \bus_1,..., \bus_\numbuses\}$ to represent a schedule $\bxhat$ (i.e., the trip sequences of each bus).

\paragraph{\underline{Step (i).}}
To create a C-MIS, we first identify one service requirement that is violated (i.e., one inequality in  \eqref{ccs:service-trips} and \eqref{ccs:service-routes}). We then create a minimal subset of trips  $\delayMIS\subseteq\delaytrips$ that violates such requirement and also satisfies a special \textit{predecessor condition} detailed in what follows. 
If we choose \eqref{ccs:service-trips}, $\delayMIS$ is the minimal number of delayed trips that violate the constraint when $z=0$, that is, $\delayMIS \subseteq \delaytrips$ such that $|\delayMIS| = \numtrips - \numserviceTrip +1$. Analogously, if the chosen condition is \eqref{ccs:service-routes} for some route $r\in \routes$, we consider $\delaytrips^r = \delaytrips\cap \trips(r)$ and $\delayMIS \subseteq \delaytrips^r$ such that $|\delayMIS| = \numtrips_r -\numserviceRoute_r +1$. 
Among these $\delayMIS$ options, we consider those that satisfy the following predecessor condition: for each trip $i \in \delayMIS$, any delayed predecessor should also be in $\delayMIS$, 
that is, if $i \in \delayMIS$ with $i \in \bus_b$ then $j \in \delaytrips\cap\bus_b$   with $y^*_j\leq y^*_i$ is also in $\delayMIS$.  The analogous condition is imposed when choosing \eqref{ccs:service-routes} for some $r\in \routes$ by replacing $\delaytrips$ with $ \delaytrips^r$. This predecessor condition is crucial to guarantee that the resulting IS is minimal, as shown in Proposition \ref{prop:greedy_mis}.

\paragraph{\underline{Step (ii).}}
Given $\delayMIS$, we create a C-MIS  $\scheduleCMIS$ by selecting a trip subsequence from the schedule $\bxhat$ for each trip in $\delayMIS$ that \textit{explain its delay}. We say that a trip subsequence of $p>1$ trips $(i_1,...i_{p-1},j)$ explains the delay of trip $j\in \delayMIS$ if $j$ would be delayed even when trip $i_1$ start as early as possible (i.e., $\stime_{i_1} - \lb$). Also, the subsequence $(i_1,...i_{p-1},j)$ is minimal if it explains the delay of $j$, but the subsequence without trip $i_1$ (i.e., $(i_2,...i_{p-1},j)$) does not.

Algorithm \ref{alg:cmis-greedy} details the procedure to build a C-MIS for a given violated constraint $\con \in \{\eqref{ccs:service-trips}, \{\eqref{ccs:service-routes}\}_{r\in \routes} \}$ and find minimal subsequences that explain the delays of trips in the corresponding $\delayMIS$. The procedure iterates over the scheduled buses and selects the last delayed trip of a bus that belongs to $\delayMIS$, trip $i\in \delayMIS$ (line 4). For trip $i$ to be delayed, its start time should be strictly greater than $\stime_i + \ub$, so we consider the earliest delayed start time to be $s_i + \ub + \tol$ (line 5), where $\tol>0$ is a predefined tolerance (e.g., $\tol= 0.1$ or $\tol=0.01$). Since a predecessor trip must cause this delay, we select the previous trip on the bus, trip $j$, add the trip pair $(j,i)$ to the C-MIS, and then check if the predecessor is the sole source of the delay (lines 7-8). If $i$ is still delayed even if $j$ starts as early as possible (i.e., $s_j -\lb$), then $j$ is sufficient to explain the delay of $i$, and there is no need to check other predecessor trips (lines 9-10). Otherwise, $j$ has to start later than $\stime_j-\lb$ to cause the delay for $i$, and we calculate the earliest possible start time of $j$ for $i$ to be delayed (lines 12). Lastly, if $j$ belongs to $\delayMIS$, we also need to explain the daily of $j$, thus, we also consider the earliest start time of $j$ (i.e., $\stime_j +\ub + \tol$) (line 14). The algorithm iterates until we explain the delay for all trips in $\delayMIS$.

\begin{algorithm}[htb] 
\small
    \linespread{1.2}\selectfont
	\caption{C-MIS construction based on the greedy algorithm solution}\label{alg:cmis-greedy}
	\begin{algorithmic}[1]
		\Procedure{\algMIS}{$s,\bus(\bxhat)$, $\by^*$, $\bv^*$, $\con \in \{\eqref{ccs:service-trips}, \{\eqref{ccs:service-routes}\}_{r\in \routes} \}$ }
            \State Build a $\delayMIS$ based on \con\ as explained in \textit{Step (i)} of Section \ref{sec:mis-greedy}
		\For {$b \in \{1,...,\numbuses\}$ such that $\bus_b\cap \delayMIS\neq \emptyset$}
		\State Select trip $i \in\bus_b\cap \delayMIS$ with the latest start time and set $\delayMIS= \delayMIS\setminus \{i\}$
            \State $\algDelayToExplain = \stime_i + \ub +\tol$
            \While {$\algDelayToExplain > 0$}
            \State Select the trip $j\in \bus_b$ that is scheduled right before trip $i$.
            \State Add trip pair to C-MIS: $\scheduleCMIS=\scheduleCMIS \cup \{(j,i)\}$
            \If {$\stime_j - \lb + \dur_j^s + \travel_{ji}^s\geq \algDelayToExplain$} {\hfill \color{gray} \# Trip $j$ explains the delay of $i$}
            \State $\algDelayToExplain = 0$
            \Else {\hfill \color{gray} \# Explain the delay of $j$}
            \State $\algDelayToExplain = \algDelayToExplain - (\dur_j^s + \travel_{ji}^s - \express_j)$.
            \If {$j \in \delayMIS$} 
            \State $ \algDelayToExplain = \max\{\algDelayToExplain, \; \stime_j + \ub + \tol \}$
            \EndIf
            \State Set  $i = j$.
            \EndIf
            \EndWhile
            \State Go back to line 4 if $\bus_b\cap \delayMIS\neq \emptyset$ 
		\EndFor
		\EndProcedure
	\end{algorithmic} 
\end{algorithm}

Example \ref{exa:cmis_greedy} illustrates some of the specific considerations of Algorithm \ref{alg:cmis-greedy}. Also, Proposition \ref{prop:greedy_mis} (proved in Appendix \ref{appendix:prop:greedy_mis-proof}) shows that the resulting set $\scheduleCMIS$ from Algorithm \ref{alg:cmis-greedy} is indeed a C-MIS for a given $\bxhat$.

\begin{example} \label{exa:cmis_greedy}
Consider the trip sequence depicted in Figure \ref{fig:cmis-greedy} . The graph illustrates the first 6 trips on the sequence, where shaded nodes correspond to delayed trips (i.e., $4, 6 \in \delayMIS$). The number above the arrows correspond to the duration of the predecessor trips and the travel time between trips for an scenario $s\in \scenarios$ (e.g., $\dur_2^s + \travel_{23}^s = 15$ for arc between trips 2 and 3). As illustrated in the graph, the number and the interval above each node trip $i \in \trips$ represents the earliest start time of such trip $y_i^*$ and its corresponding time window for it to be on-time (i.e., $[\stime_i - \lb,\stime_i + \ub]$). The numbers below the nodes are the $\algDelayToExplain$ used in Algorithm \ref{alg:cmis-greedy}. We consider $\express_i = 0$ for all $i \in \trips$, $\tol= 1$, $\lb = 1$ and $\ub = 3$ for this example.

To find the minimal subsequence that explains the delays of  trips 4 and 6 $\in \delayMIS$, we start by backtracking from trip 6. The minimum start time for this trip to be delayed is 96, thus, $\algDelayToExplain = 96$ (line 5 of Algorithm \ref{alg:cmis-greedy}). We now take its immediate predecessor, trip 5, and check if it the sole source of the delay of trip 6. Note that if trip 5 start as early as possible (i.e., $\stime_5 -\lb = 72$), then trip 6 would be on time because $y_6 =  \stime_5 - \lb + \dur_5^s + \travel_{5,6}^s = 72 + 23 = 95 \leq \stime_6 + \ub $. Therefore, trip 5 starts slightly later for trip 6 to be delay, that is, it must start at time  $\algDelayToExplain = 73$ or later (line 12 of Algorithm \ref{alg:cmis-greedy}). Once again, we take the predecessor of trip 5 (i.e., trip 4) and check if it can explain the start time of 5. Trip 4 has to start at time 59 for trip 5 to start at time 73, which also coincide with the earliest start time of trip 4 to be delay, thus,  $\algDelayToExplain = 59$ (line 14 of Algorithm \ref{alg:cmis-greedy}). Lastly, we analyze the predecessor of trip 4 (i.e., trip 3), and we see that if trip 3 starts as early as possible (i.e., $\stime_3 - \lb = 38$), then trip 4 starts at time 59, which coincides with $\algDelayToExplain$. Therefore, the subsequence  (3,4,5,6) explains the delays of trips 4 and 6 $\in \delayMIS$ at scenario $s\in \scenarios$. This subsequence is minimal because removing trip 3 from it will result of trips 4 and 6 to arrive on-time. 
\hfill $\square$
\end{example}

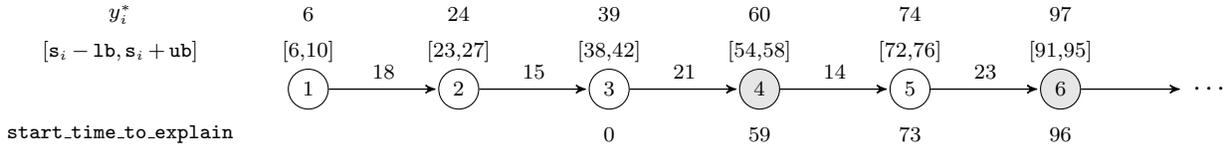
\begin{figure}[htb]
    \centering
    \begin{tikzpicture}[->,>=stealth',shorten >=1pt,auto,node distance=1cm]

\node[node trip] (n1) at (0,0) {$1$};
\node[node trip] (n2) at (2,0) {$2$};
\node[node trip] (n3) at (4,0) {$3$};
\node[node trip delayed] (n4) at (6,0) {$4$};
\node[node trip] (n5) at (8,0) {$5$};
\node[node trip delayed] (n6) at (10,0) {$6$};
\node (dots) at (12,0) {$\cdots$};

\node[tw] (tw0) at (-2.5,0.5)  {$[\stime_i - \lb,\stime_i + \ub]$};
\node[tw] (tw1) at (0,0.5)  {[6,10]};
\node[tw] (tw2) at (2,0.5)  {[23,27]};
\node[tw] (tw3) at (4,0.5)  {[38,42]};
\node[tw] (tw4) at (6,0.5)  {[54,58]};
\node[tw] (tw5) at (8,0.5)  {[72,76]};
\node[tw] (tw6) at (10,0.5) {[91,95]};

\node[tw] (y0) at (-2.5,1)  {$y^*_i$};
\node[tw] (y1) at (0,1)  {6};
\node[tw] (y2) at (2,1)  {24};
\node[tw] (y3) at (4,1)  {39};
\node[tw] (y4) at (6,1)  {60};
\node[tw] (y5) at (8,1)  {74};
\node[tw] (y6) at (10,1)  {97};

\node[tw] (e0) at (-2.5,-0.6)  {\algDelayToExplain};
\node[tw] (e3) at (4,-0.6)  {0};
\node[tw] (e4) at (6,-0.6)  {59};
\node[tw] (e5) at (8,-0.6)  {73};
\node[tw] (e6) at (10,-0.6)  {96};

\path[every node/.style={font=\scriptsize}]
(n1)
edge[arc] node [above] {18} (n2)
(n2)
edge[arc] node [above] {15} (n3)
(n3)
edge[arc] node [above] {21} (n4)
(n4)
edge[arc] node [above] {14} (n5)
(n5)
edge[arc] node [above] {23} (n6)
(n6)
edge[arc] node [above] {} (dots)
;
\end{tikzpicture}
    \vspace{-2em}
    \caption{Graphical example on how to compute a C-MIS using Algorithm \ref{alg:cmis-greedy}.}
    \label{fig:cmis-greedy}
\end{figure}

\begin{proposition}\label{prop:greedy_mis}
Consider a vehicle schedule $\bus(\xhat)$ and a scenario $\s\in \scenarios$ such that at least one of the service requirements is violated. Then, for any such constraint $\con \in \{\eqref{ccs:service-trips}, \{\eqref{ccs:service-routes}\}_{r\in \routes} \}$, 
Algorithm \ref{alg:cmis-greedy} builds a C-MIS $\scheduleCMIS$.
\end{proposition}

Algorithm \ref{alg:cmis-greedy} has two main advantages when compared to the conflict refinement procedure explained in Section \ref{sec:mis-conlfic}. First, it is a polynomial-time algorithm since it iterates two times over the set of trips  (i.e., for choosing $\delayMIS$ and for explaining their delays). Second, the procedure allows us to obtain multiple C-MISs by choosing different violated service requirements and $\delayMIS$. Our implementation takes advantage of this fact in that it creates a C-MIS for each violated service constraint (i.e.,  \eqref{ccs:service-trips} and \eqref{ccs:service-routes}), and returns their corresponding the violated cuts. Lastly, the resulting cut is equivalent to \eqref{eq:mis-cut} since we only change the procedure for obtaining a C-MIS.

As a final remark, we note that $\scheduleCMIS$ might not be an MIS for the subproblem of scenario $\s \in \scenarios$ if there are multiple service requirement violations, because it only considers a single service requirement constraint at a time. Nevertheless, our implementation constructs a C-MIS for each violated service requirement constraint for a given scenario $\s\in \scenarios$ and $\bxhat$, and we empirically observed that at least one of them is indeed an MIS. 
 
\subsection{Extended C-MIS Cuts}\label{sec:mis-strong}

We now present a procedure to find stronger cuts by extending a C-MIS. The main idea is to find alternative subsequences of trips that also explain the delay of the trips in  $\delayMIS$. Note that the trip pairs in $\scheduleCMIS$ can be represented as disjoint trip subsequences $(i_1,...,i_p)$ where $p$ is the size of a particular sequence and $(i,j)$ are consecutive trips in a subsequence if and only if $(i,j)\in\scheduleCMIS$. 

Consider a subsequence $(i_1,...,i_p)$ of $\scheduleCMIS$ with $i_p \in \delayMIS$. To create alternative subsequences that also explain the delay of $i_p$, we seek for replacements of the initial trip $i_1$ that will also explain the delay of $i_p$ and any other delayed trip such subsequence. Formally, we look for trips $i_1' \in \trips$ such that $(i_1',i_2) \in \compatible$, $i_1'$ does not appear in any trip pair in $\scheduleCMIS$, and the new sequence $(i_1', i_2,...,i_p)$ explains the delay of $i_p$ and potentially other trips in $\delayMIS$ that appear in $(i_2,..., i_{p-1})$. If any such trip $i_1'$ exists, we add the trip pair $(i_1',i_2)$ to a set of additional pairs $\scheduleMISExtra$. We then create an extended C-MIS cut that considers all trip pairs in $\scheduleCMIS \cup \scheduleMISExtra$, as shown in \eqref{eq:mis-cut-strong}. Proposition \ref{prop:extended-cmis} shows the validity of the cut and its dominance relationship to \eqref{eq:mis-cut}.

\begin{proposition} \label{prop:extended-cmis}
    Consider a vehicle schedule $\bus(\xhat)$, a scenario $\s\in \scenarios$ such that at least one of the service requirements  $\con \in \{\eqref{ccs:service-trips}, \{\eqref{ccs:service-routes}\}_{r\in \routes} \}$ is violated, and a C-MIS $\scheduleCMIS$ built using $\con$. Then, the following constraint is valid and dominates \eqref{eq:mis-cut} with $\scheduleMIS = \scheduleCMIS$, 
    \begin{equation}\label{eq:mis-cut-strong}
    \sum_{(i,j) \in \scheduleCMIS \cup \scheduleMISExtra}\sum_{k \in \depots} x_{ijk} \leq |\scheduleCMIS| - 1 +  z_\s,
    \end{equation}
    where $\scheduleMISExtra$ is constructed as previously explained. 
\end{proposition}

\begin{proof}{Proof.}
    The validity of the cut follows from the procedure for finding additional trip pairs $\scheduleMISExtra$ and that each trip has only one predecessor (i.e., inequality \eqref{cc:tripOnce}). To prove dominance, note that \eqref{eq:mis-cut-strong} and \eqref{eq:mis-cut} with $\scheduleMIS = \scheduleCMIS$ have the same right-hand side but \eqref{eq:mis-cut-strong} considers more variables by incorporating all the alternative pairs in $\scheduleMISExtra$. Thus, the extended version removes the same schedules that \eqref{eq:mis-cut}, but it also removes the alternative schedules built with $\scheduleMISExtra$. \hfill $\square$
\end{proof}

\subsection{LP-based Benders Cuts}\label{sec:lp-cuts}

We now present a procedure to find cuts for each integer master solution $\bxhat$ and an infeasible scenario $\s \in \scenarios$ using a strengthened version of the linear programming (LP) relaxation of \eqref{model:subproblem} and show that the resulting cuts are equivalent to our C-MIS cuts. 

While we can directly use the LP relaxation to derive cuts, preliminary experiments show that such relaxation is quite weak and, in most cases, does not produce a violated cut when a schedule does not meet the service requirements. Thus,  we present a procedure that tightens the coefficients of the LP relaxation and introduces additional inequalities using the optimal solution of Algorithm \ref{alg:greedy} to close the integrality gap and return valid cuts for the master problem. 

We now show how to strengthen the LP relaxation of \eqref{model:subproblem}
utilizing the optimal solution $(z^*, \by^*, \bv^*, \bu^*)$ obtained with Algorithm \ref{alg:greedy} to indirectly enforce integrality on variables $\bv$ and $z$ in its optimal solution. First, the value of $v_i$ for each $i\in \trips$ is enforced by constraint \eqref{ccs:start-ub} and is directly impacted by the big-M coefficient $\bigMotp_i$. While $\bigMotp_i$ can be arbitrarily large for the MILP formulation, we need it to be as small as possible to force $v_i=0$ whenever $y^*_i > \stime_i + \ub$ for $i \in \trips$. Thus, we set $\bigMotp_i =y^*_i - \stime_i$ for every particular $\bxhat$ and $\s$ to ensure the integrality of $v_i$ in an optimal solution of the strengthened LP. 

To enforce integrality on $z$, we include an additional constraint that relates the set of delayed trips with $z$. Specifically, we pick one constraint in $\con \in \{\eqref{ccs:service-trips}, \{\eqref{ccs:service-routes}\}_{r\in \routes} \}$ that is violated at the solution obtained by Algorithm \ref{alg:greedy} and construct a $\delayMIS$ using the procedure detailed in \textit{Step (i)} of Section \ref{sec:mis-greedy}. Using this set, we add the following inequality to the LP relaxation of the subproblem:
\begin{equation} \label{eq:extra-lp-cuts}
    \sum_{i \in \delayMIS} (1 - v_i) \leq z + |\delayMIS| - 1.
\end{equation}
Note that in any optimal solution $v_i = 0$ for all $i \in \delayMIS$ because of the tightened $\bigMotp_i$ coefficients, so \eqref{eq:extra-lp-cuts} forces $z=1$ in such solutions. 

Lastly, to guarantee the validity of the resulting cuts and prove its relationship to  \eqref{eq:mis-cut} with its corresponding $\scheduleCMIS$ (i.e., the $\scheduleCMIS$ returned by Algorithm \ref{alg:cmis-greedy} for the $\delayMIS$ use in \eqref{eq:extra-lp-cuts}), we also tighten the big-M coefficients  $\bigMstart_{ji}$ for all $(j,i)\in \compatible$ as follows:
\[ \bigMstart_{ji} = \begin{cases}
y^*_i, &  \forall (j,i) \in \schedule \\
\text{large enough number}, & \text{otherwise.} 
\end{cases}\]
Note that these tighten coefficients only appear in \eqref{ccs:start-trip} and are valid for the model. Also, they do not affect the integrality of the variables $\bv$ and $z$ in the optimal solution, but are important for our proof. To summarize, the tightened LP subproblem is given by
\begin{equation}
\min_{z \in [0,1]} \left\{ z  :\; (\bxhat, z) \in \subproblemLP^\s\cap \{ \eqref{eq:extra-lp-cuts} \}
\right\}, \tag{\textsf{SAA-Sub-LP($\s, \bxhat$)}} \label{model:subproblem-lp} 
\end{equation}
where $\subproblemLP^\s$ is the LP relaxation of $\subproblem^\s$ with the tightened $\bigMotp_i$ and $\bigMstart_{ij}$ coefficients for every $i \in \trips$ and $(i,j)\in \compatible$, respectively. Appendix \ref{appendix:lpcuts-proof} shows in detail \eqref{model:subproblem-lp} and its dual. 

Given the tightened LP model, we employ the dual of \eqref{model:subproblem-lp} to derive the following Bender's cut
\begin{equation}\label{eq:lp-cut}
    \sum_{(i,j) \in \compatible}\dualstart_{ij}\bigMstart_{ij}\sum_{k \in \depots} x_{ijk} \leq \sum_{(i,j) \in \compatible}\dualstart_{ij}\bigMstart_{ij}\sum_{k \in \depots} \xhat_{ijk} - 1 + z_\s,
\end{equation}
where $\balpha\geq 0$ are the dual variables associated with constraints \eqref{ccs:start-trip}. Proposition \ref{prop:lp-and-mis-cuts} states the validity of the cut and shows that there exists a dual solution such that \eqref{eq:lp-cut} is equivalent to the C-MIS cut \eqref{eq:mis-cut} associated with $\scheduleCMIS$. Appendix \ref{appendix:lpcuts-proof} details the proof of the proposition and shows how to construct such dual solution. 

\begin{proposition} \label{prop:lp-and-mis-cuts}
    Consider the master solution $\bxhat$ that violates one of the service requirements $\con \in \{\eqref{ccs:service-trips}, \{\eqref{ccs:service-routes}\}_{r\in \routes} \}$ for a scenario $\s \in \scenarios$, and a $\delayMIS$ constructed using $\con$. Then, there exists a dual optimal solution of \eqref{model:subproblem-lp} such that \eqref{eq:lp-cut} and the C-MIS cut \eqref{eq:mis-cut} for the corresponding $\scheduleCMIS$ are identical. 
\end{proposition}

While Proposition \ref{prop:lp-and-mis-cuts} states how to obtain a C-MIS cut using \eqref{model:subproblem-lp}, preliminary results show that it is computationally more efficient to obtain such cuts using Algorithm \ref{alg:cmis-greedy}. Therefore, we omit this alternative from our empirical results in Section \ref{sec:experiments}. Nonetheless, we believe that this insightful result could be beneficial for other CCP problems where obtaining a C-MIS (or MIS) could be computationally challenging.

\subsection{B\&C Modification Given Cuts Structure} \label{sec:b&c-zcont}

We present a modification to the B\&C procedure presented in Section \ref{sec:decomposition} that takes advantage of the specific structure of our cut variants.  
In particular, we propose relaxing the integrality constraints on $\bz$ when solving the master problem \eqref{model:master}. 
With this change, our decomposition scheme depicted in Figure \ref{fig:decomposition-diagram} enters the cut generation routine for each scenario $\s \in \scenarios$ when $\zhat_\s <1$ in the master problem, in contrast to only when $\zhat_\s =0$. Proposition \ref{prop:z-continuous} states the validity of this modification, whose proof (in Appendix \ref{appendix:prop:z-continuous-proof}) follows from the structure of our proposed cuts. 

\begin{proposition}\label{prop:z-continuous}
    The B\&C decomposition scheme with the previously mentioned modification converges in a finite number of iterations and returns an optimal schedule to \eqref{model:chance-saa}.
\end{proposition}

We note that relaxing the integrality constraints of $\bz$ is quite general and can be applied to other CCP problems where the generated cuts have a similar form to \eqref{master:cuts-added}, i.e., they are tight at the master solution producing them and enforce the integrality of $\bz$. 
Our computational results show that this modification can 
reduce the solution time and increase the number of instances solved because it significantly decreases the size of the B\&C tree by avoiding branching on $\bz$.

\section{Lagrangian-based Approach for Large-scale Instances}\label{sec:large-scale}

The branch-and-cut procedure presented in Section \ref{sec:decomposition} can obtain optimal solutions for the scenario reformulation of our problem. However, our experimental results show that it can only handle problem instances with a small number of trips in a reasonable amount of time (i.e., solve instances with 80 trips and four depots in less than two hours).
In contrast, real-world instances of the MDVSP based on small and medium-sized cities can consider hundreds or thousands of trips. 
Therefore, we need a procedure that can find vehicle schedules that satisfy the service requirements for large-size instances. 

To do so, we propose a Lagrangian-based technique that leverages our exact procedure to find low-cost vehicle schedules that satisfy the service requirements. The main idea is to partition the set of trips and create a subproblem for each partition that can be efficiently solved with our B\&C decomposition scheme. We can then devise a procedure to combine the scheduling solution of each subproblem and create a vehicle schedule that satisfies all the planning constraints of the original problem. To enforce the operational constraints modeled with the CC, we utilize Lagrangian penalties to ensure that the number of scenarios that violate the service requirements is small across all partition subproblems. As a result, our methodology finds a low-cost solution that is feasible for the original problem or violates the CC in a small amount (i.e., a few scenarios). 

Next, we detail the main components of our procedure. Section \ref{sec:ls-partition-trips} presents the proposed strategies to partition the set of trips and the resulting subproblems. Then, Section \ref{sec:ls-larg-decomposition} introduces our Lagrangian decomposition scheme to enforce the CC, and Section \ref{sec:ls-larg-dual} shows how we solve the resulting Lagrangian dual problem and obtain a vehicle schedule with the required characteristics.

\subsection{Trips Partitioning and Subproblems} \label{sec:ls-partition-trips}

Our procedure starts by partitioning the set of trips $\trips$ into small-size groups to create one CC-MDVSP subproblem for each partition. Constructing a trip partition that results in a minimum-cost schedule is a type of bin-packing problem and, as such, an NP-hard problem. Thus, this work considers a heuristic approach for partitioning trips that lead to low-cost vehicle schedules.

Since the trip partition directly affects the cost of the resulting schedule, we want to partition the trips such that several trips in a group can be scheduled on the same bus to minimize the number of buses in the scheduling solution of a subproblem (i.e., the largest cost component of the problem). To do so, we solve a deterministic MDVSP (e.g., using average times and ignoring service requirements) and use the resulting solution to assign trips to the different groups such that trips in the same bus are assigned to the same group. Algorithm \ref{alg:partition-trips} details this assignment where $\groupsize$ is the minimum size of a group, $\bus(\bxhat)$ is the scheduling solution of the deterministic MDVSP, and  $\tripgroups= \{\trips_1, ... ,\trips_\numpartition\}$ is the trip partition, where $\partitions=\{1,...,\numpartition\}$ is the set of indices of the partition of size $\numpartition=|\partitions|$. Other heuristics or exact strategies can also be considered, which we leave as future work directions. 

\begin{algorithm}[b] 
\small
    \linespread{1.2}\selectfont
	\caption{Trip partitioning based on deterministic solution}\label{alg:partition-trips}
	\begin{algorithmic}[1]
		\Procedure{\algPartition}{$\bus(\bxhat)$, $\groupsize$}
            \State $\tripgroups = \trips^{\texttt{aux}} = \emptyset$
		\For {bus $\bus \in \bus(\bxhat)$ }
		    \State Add trips to the current group:  $\trips^{\texttt{aux}} = \trips^{\texttt{aux}} \cup \bus$
		    \If {$\trips^{\texttt{aux}} \geq \groupsize$} 
            \State Add group to trip partition: $\tripgroups = \tripgroups \cup \trips^{\texttt{aux}}$, 
            and reset group:  $\trips^{\texttt{aux}} = \emptyset$
            \EndIf
		\EndFor
    \Return trip partition $\tripgroups$
	\EndProcedure
	\end{algorithmic} 
\end{algorithm}

Algorithm \ref{alg:partition-trips} creates a partition of $\trips$ such that all trips are assigned to a single group. Then, our Lagrangian-based procedure uses $\tripgroups$ to create a subproblem for each group $p \in \partitions$ given by
\begin{subequations}
\begin{align}
    \min \; & \sum_{k \in \depots}\sum_{(i,j)\in \arcs^p}\cost_{ij} x^p_{ijk}  \tag{$\textsf{SAA-MCF}^p$} \label{model:chance-saa-group} \\
    \mbox{s.t.}\;
    & \bx^p \in \mcfset^p \nonumber \\
    & \sum_{\s \in \scenarios} z^p_\s \leq  \lfloor \numscenarios \epsilon \rfloor , \label{ccs-p:probability}\\
    & (\bx^p, z^p_\s) \in \subproblem^{p,\s}, & \forall \s \in \scenarios, \label{ccs-p:scenario}\\
    & z^p_s \in \B & \forall \s \in \scenarios.  \nonumber
\end{align}
\end{subequations}
This formulation is identical to \eqref{model:chance-saa} but adjusts the variables and constraints to consider only the trips in $\trips^p$. Specifically, we use superscript $p$ to denote the modified elements that only consider trips in $\trips^p$, that is, the compatible set $\compatible^p$, the set of arcs $\arcs^p$ and the operational constraints $\subproblem^{p,\s}$ for each scenario $\s\in \scenarios$, also add it to the decision variables to distinguish the subproblems. Note that sets $\depots$ and $\scenarios$ are the same for each subproblem.

Once each subproblem is solved, we combine the schedules of \eqref{model:chance-saa-group} for each  $p \in \partitions$ to create a vehicle schedule that satisfies all the planning constraints $\mcfset$ of the original problem. Given that trip groups are pair-wise disjoint, constraints  $\mcfset^p$ for each $p \in \partitions$ will enforce the trips' coverage and the flow balance constraints (i.e., \eqref{cc:tripOnce}, \eqref{cc:flowBalance}-\eqref{cc:consistenDepot} considering all trips). The only constraint that might be violated is the maximum capacity of the depots \eqref{cc:capacity}. In such a case, we heuristically reassign the depot for some trip sequences in order to fulfill the capacity requirement. Specifically, we iteratively seek for a trip sequence $\bus_{b}=(i_1,...,i_n)$ (where $n$ is the number of trips) assigned to an overloaded depot $k'$ and reassign it to a depot $k''$ with available capacity that minimized the cost of the depot reallocation, that is,   $k''\in \argmin_{k \in \depots\setminus\{k'\} }\{ c_{ki_1} + c_{i_n k} : k \text{ has available capacity} \}$.

Therefore, we can use this procedure (i.e., partition the trips, solve the subproblems, and reassign the depots) to create a vehicle schedule that fulfills all the operational constraints of the problem. Unfortunately, this technique will not necessarily enforce the CC of the original problem because \eqref{model:chance-saa-group} enforces individual CCs for each partition $p \in \partitions$ and, thus, ignores the joint behavior. The following section addresses this issue with a Lagrangian decomposition approach.

\subsection{Lagrangian Decomposition} \label{sec:ls-larg-decomposition}


A reason why satisfying the individual CC for each subproblem \eqref{model:chance-saa-group} does not enforce the common CC of the original problem is that each subproblem can unmet the service requirements in different scenarios. This discrepancy usually results in more scenarios with unmet service conditions for the combined vehicle schedule and, thus, a violation of the common CC. Inspired by this observation, we propose a Lagrangian decomposition scheme that limits the set of scenarios with unmet service requirements across all subproblems via Lagrangian penalties. Specifically, we propose model \eqref{model:chance-saa-joint} that combines subproblems \eqref{model:chance-saa-group} for each $p \in \partitions$ and relates the scenario variables for each partition. 
\begin{subequations}
\begin{alignat}{2}
    \min \; & \sum_{p \in \partitions}\sum_{k \in \depots}\sum_{(i,j)\in \arcs^p}\cost_{ij} x^p_{ijk}  \qquad \tag{$\textsf{SAA-MCF-Joint}$} \label{model:chance-saa-joint} \\
    \mbox{s.t.}\;
    & \bx^p \in \mcfset^p,  && \forall p \in \partitions \nonumber   \\
    & \sum_{\s \in \scenarios} z^p_{\s} \leq  \lfloor \numscenarios \epsilon \rfloor  && \forall p \in \partitions,  \label{ccs-lag:min-violation}  \\
    & (\bx^p, z^p_{\s}) \in \subproblem^{p,\s} && \forall \s \in \scenarios, p \in \partitions,  \nonumber \\
    & z^1_{\s} \geq \frac{1}{\numpartition-1}\sum_{p = 2}^\numpartition z^p_{\s} &&  \forall \s \in \scenarios, \label{ccs-lag:common} \\
    & z^p_{\s} \in \B && \forall \s \in \scenarios, p \in \partitions.  \nonumber
\end{alignat}
\end{subequations}

This model includes all the constraints and variables of \eqref{model:chance-saa-group} for each group and links the scenario variables through constraint \eqref{ccs-lag:common}, where $z_{\s p}$ is the indicator variable associated with scenario $\s\in \scenarios$ and group $p \in \partitions$. Specifically, \eqref{ccs-lag:common} forces a scenario variable in the first group to be turn-on if the service requirements are unmet in any other group. Since the number of scenarios in the first group with unmet service conditions is restricted to be less than $\lfloor \numscenarios \epsilon \rfloor$ \eqref{ccs-lag:min-violation}, we guarantee that the total number of scenarios with unmet service requirements across all subproblems is also bounded by that amount.

We develop a Lagrangian decomposition scheme based on \eqref{model:chance-saa-joint} that dualizes \eqref{ccs-lag:common} to create sub-problems for each trip group. The resulting objective function is given by:
\[  \sum_{p \in \partitions}\sum_{k \in \depots}\sum_{(i,j)\in \arcs^p}\cost_{ij} x^p_{ijk}  + \sum_{\s \in \scenarios}\mu_\s \left( -(\numpartition - 1) z^1_{\s} + \sum_{p=2}^\numpartition z^p_{\s} \right), \]
where $\mu_\s\geq 0$ for $\s \in \scenarios$ are the Lagrangian penalties associated with \eqref{ccs-lag:common}. Thus, the resulting Lagrangian dual problem is
\begin{equation}\label{model:lagragiandual}
\max_{\bmu}\left\{ \sum_{p \in \partitions} \lagrsub^p(\bmu):\; \bmu \in \R^{\numscenarios}_+  \right\} \tag{\textsf{LagrDual}}
\end{equation}
where we have one Lagrangian subproblem for each group $p \in \partitions$ given by
\[
\lagrsub^p(\bmu) = \min_{\bx,\bz} \left\{ \sum_{k \in \depots}\sum_{(i,j)\in \arcs^p}\cost_{ij} x_{ijk}  + \lagrconst^p\sum_{\s \in \scenarios}\mu_\s z_\s  :\;  \bx \in \mcfset^p, \eqref{ccs-p:probability}-\eqref{ccs-p:scenario}, \bz \in \B^\numscenarios
\right\}
\]
with constant $\lagrconst^1 = -(\numpartition - 1)$ and $\lagrconst^p = 1$ for all $p \in \partitions\setminus \{1\}$.

This decomposition allows us to find vehicle schedules with a small number of scenarios with unmet service requirements by solving the Lagrangian dual problem. However, this procedure might not necessarily enforce the CC in \eqref{model:chance-saa}. In particular, model  \eqref{model:chance-saa-joint} enforces the service requirements for each trip group (i.e., constraints in set $\subproblem^{p,\s}$), but the service requirements in \eqref{model:chance-saa} (i.e., $\subproblem^\s$) consider all trips together. Thus, there could be a scenario where the service requirements are met for all groups individually, but the service requirements for all trips together might be unmet. Our preliminary experiments show that this behavior could happen but usually leads to only a slight violation of the CC.

\subsection{Solving the Lagrangian Dual Problem} \label{sec:ls-larg-dual}

There is a vast literature on methodologies to solve Lagrangian dual problems, such as sub-gradient and cutting plane method \citep{fisher2004lagrangian,frangioni2005lagrangian}. Our problem \eqref{model:lagragiandual} has the particularity that each subproblem $\lagrsub^p(\bmu)$ for $p \in \partitions$ is an NP-hard problem that can take a significant amount of time to solve due to the CC. Therefore, we opt for the bundle method \citep{Lemarechal1975,frangioni2002generalized} to solve \eqref{model:lagragiandual} since it has a fast converge rate (i.e., $\mathcal{O}(\frac{1}{\epsilon^3})$ for a tolerance $\epsilon>0$) when compared to other alternatives, and thus, requires fewer iterations to converge to the optimal solutions. 

A bundle method is a variant of the cutting plane method, which adds a quadratic stabilizer to improve convergence. The procedure solves subproblems $\lagrsub^p(\hat{\bmu})$ for each $p \in \partitions$ for a given $\hat{\bmu} \geq 0$ and creates a joint subgradient associated to the optimal solution $(\hat{\bx}^p,\hat{\bz}^p)$ of each subproblem as:
\[  \grad(\hat{\bmu})^\top = \left( -(\numpartition - 1) \hat{z}^1_{1} + \sum_{p=2}^\numpartition \hat{z}^p_{1}, \; \dots \;, -(\numpartition - 1) \hat{z}^p_{\numscenarios} + \sum_{p=2}^\numpartition \hat{z}^p_{\numscenarios}  \right) \]
This subgradient is then used to create a cutting plane valid for $\sum_{p \in \partitions}\lagrsub^p(\cdot)$, which is added to a quadratic optimization problem that returns the next set of Lagrangian penalties. Specifically, we start with $\bmu_0 = (0,\dots, 0)$ and in each iteration $\ell\geq 0$, we first solve the subproblems $\lagrsub^p(\bmu^\ell)$ for the current $\bmu^\ell$, create subgradient $\grad(\bmu^\ell)$, and add a cutting plane to the following quadratic problem that finds the new set of Lagrangian penalties:
\begin{align*} 
    \bmu^{\ell+1} \in \argmax_{\bmu \in \mathbb{R}^{\numscenarios}_{+}} \bigg \{\; &  \theta + \frac{1}{2t}|| \bmu - \bmu^{\ell} || : \;
    \theta \leq \sum_{p \in \partitions}\lagrsub^p(\bmu^l) + \grad(\bmu^l)^\top(\bmu - \bmu^l), \; \forall l\in \{0,..., \ell\}, \theta \in \mathbb{R} \bigg \}.
\end{align*}  

In this problem, variable $\theta$ over-approximates the optimal value of \eqref{model:lagragiandual}, that is, a valid dual bound for the problem. The objective function includes a quadratic stabilizer $\frac{1}{2t}|| \bmu - \bmu^{\ell} ||$ to improve converge,  where $t<1$ is a positive parameter that is updated in each iteration as suggested by \cite{Lemarechal1975}. The procedure ends when we observe a small relative difference between the Lagrangian primal and dual bound (i.e., $\sum_{p \in \partitions}\lagrsub^p(\bmu^\ell)$ and the optimal value of $\theta$ for iteration $\ell$, respectively) or if there is small relative difference between the primal bounds or dual bounds from one iteration to the next (we use a tolerance of 0.001). In addition, we define a maximum number of iterations (i.e., $\ell\leq 100$) due to the expensive computational cost of solving each iteration. However, our experiments show that in most cases the algorithm converges before this limit is achieved. 

Our procedure also builds a solution for \eqref{model:chance-saa} in each iteration (if possible) and saves the best solution found so far. Specifically, we create a vehicle schedule using the solutions of each subproblem and adjust the depot assignments when necessary, as detailed in Section \ref{sec:ls-partition-trips}. We then evaluate the service requirements $\subproblem^\s$ for each scenario $\s\in \scenarios$ and check if constraint \eqref{ccs:probability} is satisfied or not. We first keep the solutions with the smallest violations of \eqref{ccs:probability} (i.e., fewer scenarios with violated service requirements). If a feasible solution is found, we then only keep the feasible schedules with the lowest cost. We return the best schedule found so far at the end of the procedure. By doing so, we can guarantee that the resulting schedule is either feasible with the smallest cost found so far or infeasible with the smallest violation of the CC.


\section{Empirical Evaluation} \label{sec:experiments}


We now present the numerical experiment results for the B\&C scheme and the Lagrangian-based approach presented in Sections \ref{sec:decomposition} and \ref{sec:large-scale}, respectively, to solve the CC-MDVSP. 
In what follows, we first provide details on the test instances and the experimental setup. We then present experimental results comparing the performances of different cut types for our B\&C procedure and show the value of the CC variant when compared to the deterministic version of the problem used in practice. Finally, we evaluate the solution quality of our Lagrangian-based approach for large-scale instances and compare them to the solutions found by the deterministic version.

\subsection{Experimental Setup} \label{sec:experiments-setup}
We generate random instances of the problem following the procedure described in \citep{carpaneto1989branch,kulkarni2019benchmark} with slight modifications to incorporate the stochastic travel times, expressing, and group trips into routes (see 
Appendix \hyperlink{appendixB}{B} for further details).  We consider instances with $\numtrips\in \{50,60,70,80\}$ trips, 10 trips per route, and $\numdepots \in \{2,3,4\}$ depots. We generate five instances per each $\numtrips$ and $\numdepots$ configuration for a total of 60 randomly generated instances. All instances, unless specify otherwise, utilize  $\epsilon=0.05$, $\serviceTrips = 0.9$, and $\serviceRoute = 0.8$  for the CC probability and the service requirements, respectively. We consider 750 scenarios to solve the SAA formulation of the problem and independently sample 2000 scenarios to evaluate the quality of the solutions. We preliminary tested with different values for $S \in \{500, 750, 1000\} $ and concluded that $\numscenarios = 750$ was suitable for our experiments. Lastly, as commonly assumed in the literature (see, for example, \cite{kulkarni2018new, ricard2022increasing}, we consider all travel times to be integer values, thus, we use $\tol = 1$ when computing our C-MIS cuts. 


The code for the decomposition algorithms is implemented in \texttt{C++} using solver IBM ILOG CPLEX 20.1 with callback functions. All experiments are run on a single thread with a 16GB memory limit using the University of Toronto SciNet server Niagara\footnote{See \url{https://docs.scinet.utoronto.ca/index.php/Niagara_Quickstart} for further server specifications.}. 



\subsection{Efficiency and Value of the B\&C Decomposition Scheme}

To evaluate our B\&C decomposition scheme to solve \eqref{model:chance-saa}, we compare the different cut generation procedures introduced in Section \ref{sec:cutgeneation} and evaluate the impact of other algorithmic enhancements, such as the valid inequalities presented in Section \ref{sec:master-valid-inequalities}. We also assess the quality of the return schedules and show that the solutions are more robust and have a lower cost than the ones obtained by solving the deterministic MDVSP with over-approximations of the traveling times, a practice commonly used by transit agencies \citep{kittelson2003transit,furth2007service}. In these experiments, we use a time limit of 2 hours per instance.

\paragraph{\underline{Cut Comparison.}} We first compare the computational performance of the cut generation alternatives:   MIS cuts \eqref{eq:mis-cut} using the conflict refinement tool of CPLEX (i.e., \misConflict), the C-MIS cuts and its extended version  (i.e., \misGreedy\ and \misGreedyEnhance, respectively).
Table \ref{tab:b&c-cuts} reports the number of problems solved to optimality, average optimality gap, solving time, and added cuts. Optimality gaps are the ones returned by the solver, and the average gap only considers the instances without optimality proof. The average time and added cuts columns only consider the cases where all alternatives found an optimal solution. Bold numbers correspond to the best values, that is, the largest values in a row for the first set of columns and the smallest values in the remaining set of columns.

\begin{table}[htbp]
  \centering
  \small
  \caption{Cut generation comparison in B\&C approach}
  \begin{adjustbox}{width=\textwidth,center}
    \begin{tabular}{cc|rrr|rrr|rrr|rrr}
    \toprule
          &       & \multicolumn{3}{c|}{\# Optimal} & \multicolumn{3}{c|}{\% Optimality Gap} & \multicolumn{3}{c|}{Av. Time (sec)} & \multicolumn{3}{c}{Av. Added Cuts} \\
    \midrule
    $\numtrips$ & $\numdepots$ & \misConflict    & \misGreedy     & \misGreedyEnhance   & \misConflict    & \misGreedy     & \misGreedyEnhance   & \misConflict    & \misGreedy     & \misGreedyEnhance  & \misConflict    & \misGreedy     & \misGreedyEnhance \\
    \midrule
        \multirow{3}[2]{*}{50} & 2     & \textbf{5} & \textbf{5} & \textbf{5} & -  & -  & - & 45.9 & 15.8 & \textbf{12.3}  & 1687 & \textbf{1207}  & 1254 \\
          & 3     & \textbf{5} & \textbf{5} & \textbf{5} & -     & -     & -     & 7.1		 & 2.9 & \textbf{1.7} & \textbf{378} & 633  &  412 \\
          & 4     & \textbf{5} & \textbf{5} & \textbf{5} & -     & -     & -     & 32.1		 & \textbf{10.4} & 13.1  & 894  & 804  & \textbf{799} \\
    \midrule
    
    \multirow{3}[2]{*}{60} & 2     & 4     & 4     & 4     & 0.28  & 0.18  & \textbf{0.10} & 48.5  & \textbf{14.4} & 14.6  & 1538  & 1388  & \textbf{1090} \\
          & 3     & \textbf{5} & \textbf{5} & \textbf{5} & -     & -     & -     & 806.4 & 821.8 & \textbf{644.5} & 3257  & 2980  & \textbf{2652} \\
          & 4     & \textbf{5} & \textbf{5} & \textbf{5} & -     & -     & -     & 530.6 & 458.7 & \textbf{317.7} & 2382  & 2508  & \textbf{2300} \\
    \midrule
    \multirow{3}[2]{*}{70} & 2     & 3     & 3     & 3     & 5.37  & \textbf{0.69} & 3.05  & 814.0 & \textbf{256.5} & 632.2 & 3797  & \textbf{3589} & 3859 \\
          & 3     & \textbf{5} & \textbf{5} & \textbf{5} & -     & -     & -     & 941.2 & 965.4 & \textbf{558.3} & 3609  & 3570  & \textbf{2321} \\
          & 4     & \textbf{3} & 2     & \textbf{3} & 0.26  & \textbf{0.17} & 0.21  & \textbf{1291.7} & 1712.4 & 1569.4 & \textbf{2456} & 3201  & 2826 \\
    \midrule
    \multirow{3}[1]{*}{80} & 2     & \textbf{3} & \textbf{3} & \textbf{3} & 0.84  & 0.49  & \textbf{0.42} & 324.9 & 249.9 & \textbf{76.3} & 2661  & 2904  & \textbf{1655} \\
          & 3     & 3     & 3     & \textbf{4} & \textbf{4.53} & 5.14  & 5.26  & 125.8 & 111.7 & \textbf{77.7} & 1578  & 1886  & \textbf{1630} \\
          & 4     & \textbf{3} & \textbf{3} & \textbf{3} & 5.40  & 11.75 & \textbf{4.88} & \textbf{693.9} & 1293.2 & 1383.3 & 3057  & \textbf{2941} & 3246 \\
    \midrule
    \multicolumn{2}{c|}{Total/Av.} & 49    & 48    & \textbf{50} & 3.01  & 3.07  & \textbf{2.25} & 426.4 & 429.0 & \textbf{363.4} & 2006  & 1998  & \textbf{1743} \\
    \bottomrule
    \end{tabular}%
    \end{adjustbox}
  \label{tab:b&c-cuts}%
\end{table}%

Table \ref{tab:b&c-cuts} shows that \misGreedyEnhance\ generally performs best by solving the largest number of instances and obtaining the smallest optimality gaps and solving times. We observe that \misConflict\ and \misGreedy\ have similar performances, which is mostly explained because the conflict refinement of CPLEX was run only after our greedy algorithm determined that a cut was needed. Also, it is interesting to see that \misConflict\ and \misGreedy\ add similar number of cuts, which shows the benefit of adding multiple \misGreedy\ per violated subproblem and, thus, reduce the search space. Lastly, \misGreedyEnhance\ adds the smallest amount of cuts in general, which is expected due to the dominance relationship with \misGreedy.


\paragraph{\underline{Enhancement Comparison.}} The above results considered enhancements to the basic B\&C algorithm, in particular: (i) the valid inequalities in Section \ref{sec:master-valid-inequalities}, and (ii) relaxing the integrality of the $\bz$ variables in Section \ref{sec:b&c-zcont}. Table \ref{tab:b&c-enhacements} shows the importance of both enhancements in our procedure when using \misGreedyEnhance\ cuts (we observe similar results for other cut variants). Here, \none\ avoids both enhancements, \validInequalities\ only uses the valid inequalities, \zcont\ considers $\bz$ as continuous variables, and \both\ uses both enhancements. The meaning of the columns and bold numbers is the same as in Table~\ref{tab:b&c-cuts}. In particular, the average time columns only consider the instances that are solved to optimality across all four variants and, thus, differ from the values in Table \ref{tab:b&c-cuts}. 

\begin{table}[htbp]
  \centering
  \small
  \caption{B\&C enhancement comparison with cuts \misGreedyEnhance}
    \begin{tabular}{cc|rrrr|rrrr|rrrr}
    \toprule
          &       & \multicolumn{4}{c|}{\# Optimal} & \multicolumn{4}{c|}{\% Optimality Gap} & \multicolumn{4}{c}{Av. Time (sec)} \\
    \midrule
    $\numtrips$ & $\numdepots$ & \none  & \validInequalities    & \zcont    & \both  & \none  & \validInequalities    & \zcont    & \both  & \none  & \validInequalities    & \zcont    & \both    \\
    \midrule
    \multirow{3}[2]{*}{50} & 2     & 5     & 5     & 5     & 5     & -     & -     & -     & -     & 12.6  & 33.7  & \textbf{9.2} & 12.3 \\
          & 3     & 5     & 5     & 5     & 5     & -     & -     & -     & -     & 1.7   & 10.4  & \textbf{1.5} & 1.7 \\
          & 4     & 5     & 5     & 5     & 5     & -     & -     & -     & -     & 8.8   & 23.1  & \textbf{5.8} & 13.1 \\
    \midrule
    \multirow{3}[2]{*}{60} & 2     & 4     & 4     & 4     & 4     & 0.27  & 0.25  & 0.23  & \textbf{0.10} & 67.4  & 48.7  & 18.9  & \textbf{14.6} \\
          & 3     & 4     & 4     & 4     & \textbf{5} & 0.38  & 0.17  & 0.17  & -     & 387.7 & 244.2 & 335.8 & \textbf{165.2} \\
          & 4     & \textbf{5} & 4     & \textbf{5} & \textbf{5} & -     & 0.05  & -     & -     & 608.2 & 532.9 & 324.1 & \textbf{287.1} \\
    \midrule
    \multirow{3}[2]{*}{70} & 2     & \textbf{3} & 2     & \textbf{3} & \textbf{3} & 5.68  & 5.24  & 3.06  & \textbf{3.05} & 18.1  & 140.2 & \textbf{13.9} & 23.3 \\
          & 3     & \textbf{5} & 4     & \textbf{5} & \textbf{5} & -     & 0.05  & -     & -     & \textbf{59.4} & 1839.0 & 70.3  & 68.6 \\
          & 4     & 2     & 2     & 2     & \textbf{3} & 0.20  & 0.91  & \textbf{0.15} & 0.21  & \textbf{1565.1} & 1923.2 & 1584.2 & 1569.4 \\
    \midrule
    \multirow{3}[2]{*}{80} & 2     & \textbf{3} & 2     & \textbf{3} & \textbf{3} & 0.53  & 0.62  & 0.64  & \textbf{0.42} & \textbf{4.0} & 13.9  & 5.7   & 9.1 \\
          & 3     & 3     & 3     & 3     & \textbf{4} & \textbf{2.83} & 10.56 & 4.97  & 5.26  & 37.7  & 180.7 & \textbf{32.1} & 77.7 \\
          & 4     & \textbf{3} & 2     & \textbf{3} & \textbf{3} & 7.20  & 5.68  & 4.97  & \textbf{4.88} & \textbf{39.1} & 355.0 & 74.2  & 100.6 \\
    \midrule
    \multicolumn{2}{c|}{Total/Average} & 47    & 42    & 47    & \textbf{50} & 2.60  & 3.28  & \textbf{2.17} & 2.25  & 189.8 & 390.5 & 155.5 & \textbf{140.8} \\
    \bottomrule
    \end{tabular}%
  \label{tab:b&c-enhacements}%
\end{table}%

Table \ref{tab:b&c-enhacements} shows that there is a clear advantage of using both enhancements to improve all metrics, and, surprisingly, this effect does not occur when we try each enhancement on its own. In particular, we see that \validInequalities\ has a negative effect in terms of run time and instances solved, mostly due to the large number of inequalities included (i.e., at most one per scenario and service requirement, for a total of $2000\sim 5000$ inequalities) and, thus, enlarging the master problem. However, when combined with \zcont, the effect of additional constraints is significantly decreased since we are reducing the branching factor of the problem by ignoring the integrality of $\bz$. We also try other enhancements presented in the literature for the deterministic MCF model (i.e., odd-cycle inequalities and variable fixing, as shown by  \cite{hadjar2006branch} and \cite{groiez2013separating}), but our results show that these alternatives have none or negative effects in the procedure, which might be explained by the different problem structure between the deterministic MDVSP and the CC-MDVSP.

\paragraph{\underline{Solution Quality.}} Lastly, we compare the quality of the solutions found with our best CC variant (i.e., \chanceSAA, using \misGreedyEnhance\ and \both) with two deterministic variants commonly used by practitioners \citep{kittelson2003transit,furth2007service}: (i) a model that considers only average travel times (i.e., \mean), and (ii) a conservative approach that uses the 75 percentile to estimate travel times (i.e., \percentile). Table \ref{tab:b&c-quality} shows the results where the first set of columns presents the average objective value found by each procedure, the second set is the relative difference between models \chanceSAA\ and \percentile\ with respect to \mean\ (i.e., $\frac{X- \mean}{\mean}$ with $X \in \{\chanceSAA, \percentile\}$), and the third set presents the percentage of scenarios that satisfy the CC when evaluated over 2000 scenarios that were not considered during the optimization phase. 

\begin{table}[htbp]
  \centering
  \small
  \caption{Solution quality comparison with deterministic approaches}
    \begin{tabular}{cc|rrr|rr|rrr}
    \toprule
          &       & \multicolumn{3}{c|}{Objective value} & \multicolumn{2}{c|}{\% Obj. Diff.} & \multicolumn{3}{c}{\% Scenarios} \\
    \midrule
    $\numtrips$ & $\numdepots$ & \mean  & \percentile & \chanceSAA    & \percentile & \chanceSAA    &  \mean  & \percentile & \chanceSAA \\
    \midrule
    \multirow{3}[2]{*}{50} & 2     &     161,110.4  &     172,332.4  &     161,250.4  & 7.0   & 0.1   & 83.9  & 99.4  & 94.4 \\
          & 3     &     170,010.4  &     179,480.0  &     170,054.4  & 5.6   & 0.0   & 87.8  & 99.1  & 95.5 \\
          & 4     &     168,266.8  &     179,713.2  &     168,357.6  & 6.8   & 0.1   & 87.2  & 99.3  & 95.1 \\
    \midrule
    \multirow{3}[2]{*}{60} & 2     &     217,675.6  &     239,556.4  &     217,939.7  & 10.1  & 0.1   & 82.2  & 98.2  & 94.5 \\
          & 3     &     201,747.2  &     221,111.6  &     202,062.4  & 9.6   & 0.2   & 76.2  & 98.1  & 94.3 \\
          & 4     &     225,984.0  &     241,715.2  &     226,221.5  & 7.0   & 0.1   & 86.1  & 98.2  & 95.0 \\
    \midrule
    \multirow{3}[2]{*}{70} & 2     &     225,623.2  &     250,116.4  &     229,758.8  & 10.9  & 1.8   & 74.3  & 97.7  & 93.7 \\
          & 3     &     226,316.4  &     246,416.4  &     226,497.8  & 8.9   & 0.1   & 79.7  & 98.2  & 93.4 \\
          & 4     &     232,159.2  &     256,424.4  &     232,975.2  & 10.5  & 0.4   & 71.3  & 97.5  & 93.2 \\
    \midrule
    \multirow{3}[2]{*}{80} & 2     &     277,649.6  &     300,261.2  &     278,568.3  & 8.1   & 0.3   & 76.1  & 97.3  & 93.0 \\
          & 3     &     259,916.4  &     280,397.6  &     264,609.8  & 7.9   & 1.8   & 74.3  & 97.5  & 93.8 \\
          & 4     &     245,881.2  &     266,018.4  &     251,442.8  & 8.2   & 2.3   & 72.4  & 94.2  & 93.4 \\
    \midrule
    \multicolumn{2}{c|}{Average:} &       &       &       & 8.4   & 0.6   & 79.3  & 97.9  & 94.1 \\
    \bottomrule
    \end{tabular}%
  \label{tab:b&c-quality}%
\end{table}%

We observe that \chanceSAA\ obtains a very low objective value, usually no more than 1\% over \mean, and has a percentage of scenarios that satisfy the CC close to the target number of 95\%. In contrast, \mean\ has the lowest cost but only fulfills the OTP conditions in 79\% of the scenario on average, and  \percentile\ significantly increases the objective value and usually returns schedules that are too conservative. Therefore, our \chanceSAA\ alternative is the best in terms of finding low-cost schedules that can achieve OTP in the range of the number of scenarios desired (i.e., 95\%).  

\subsection{Performance of the Lagrangian-based Approach}

To test the Lagrangian-based approach for large-scale instances, we use instances with $\numtrips \in \{200, 300\}$, and all the other parameters remain the same. These experiments consider a total time limit of 24 hours, with a 10-minute time limit per subproblem and a maximum of 100 Lagrangian dual iterations. None of our experiments reached the limits for the total time and Lagrangian iterations. Lastly, we solve each subproblem using the best cut generation and enhancement combination shown in the previous section (i.e., $\misGreedyEnhance$ with $\both$).

We first evaluate the quality of the solutions for different group sizes where each group $p \in \partitions$ had a maximum of $\numtrips_p \in \{20,30,40,50\}$ trips, (i.e., 4 to 15 groups). As explained in Section \ref{sec:ls-partition-trips}, we solve the deterministic MDVSP (in this case \percentile) to group the trips and ensure that each group has many trips that can be scheduled together. Preliminary results using \mean\ and random alternatives show significantly worse performance. Table \ref{tab:lagr-solving} shows the results for the Lagrangian-based approach for the four group sizes (i.e., \La{20}, \La{30}, \La{40}, and \La{50}) when solving for 750 scenarios. The first set of columns shows the average objective value. The second set of columns presents the percentage of scenarios that fulfill the service requirements. The last set of columns shows the total run time in hours. Bold numbers correspond to the lowest values for column sets 1 and 3 in each row, and values between $95 \pm 1\%$ for column set 2.

\begin{table}[htbp]
  \centering
  \small
    \caption{Performance comparison for Lagrangian-based approach with different group sizes}
    \begin{tabular}{cc|rrrr|rrrr|rrrr}
    \toprule
          &       & \multicolumn{4}{c|}{Objective Value ($\times10^3$)} & \multicolumn{4}{c|}{\% Scenarios OTP} & \multicolumn{4}{c}{Time (hours)} \\
    \midrule
    $\numtrips$ & $\numdepots$ &\La{20}  & \La{30}  & \La{40}  & \La{50} & \La{20}  & \La{30}  & \La{40}  & \La{50} & \La{20}  & \La{30}  & \La{40}  & \La{50} \\
    \midrule
    \multirow{3}[1]{*}{200} & 2     &    658.7  &         657.9  &   647.8  & \textbf{  639.3} & \textbf{95.5} & \textbf{95.7} & \textbf{95.4} & \textbf{95.0} & \textbf{  1.12} &      3.45  &      1.61  &      2.72  \\
          & 3     &    641.5  &         635.9  &   633.2  & \textbf{  623.4} & 96.1  & \textbf{95.7} & \textbf{95.9} & \textbf{94.6} & \textbf{  1.40} &      4.37  &      5.22  &      6.45  \\
          & 4     &    638.2  &         637.0  &   631.1  & \textbf{  622.5} & 96.1  & \textbf{95.7} & \textbf{95.5} & \textbf{95.5} & \textbf{  1.32} &      2.88  &      4.63  &      6.84  \\
    \midrule
    \multirow{3}[2]{*}{300} & 2     &    959.4  &         948.7  &   933.8  & \textbf{  922.1} & \textbf{94.6} & \textbf{94.9} & 92.7  & 91.4  & \textbf{  2.31} &      5.45  &   10.59  &   16.21  \\
          & 3     &    973.2  &         954.0  &   951.6  & \textbf{  932.8} & \textbf{95.8} & 93.9  & 93.4  & 93.0  & \textbf{  5.94} &   12.07  &   13.53  &   19.40  \\
          & 4     &    968.4  &         957.0  &   943.7  & \textbf{  936.7} & \textbf{94.9} & 92.9  & 92.8  & 92.7  & \textbf{  9.01} &   14.98  &   21.53  &   16.51  \\
    \midrule
    \multicolumn{2}{c|}{Average} &    806.6  &         790.2  &   779.5  & \textbf{  774.5} & \textbf{95.5} & \textbf{94.3} & 93.7  & 92.4  & \textbf{  3.52} &      9.52  &   11.35  &   10.66  \\
    \bottomrule
    \end{tabular}%
  \label{tab:lagr-solving}%
\end{table}%

We observe intuitive results for instances with $\numtrips = 200$, in which \La{50} achieves the lowest cost schedule satisfying the CC to a reasonable degree (i.e., $95 \pm 0.5\%$ ), but taking significantly more time than smaller group alternatives. In contrast, when $\numtrips = 300$, only \La{20} finds schedules that fulfill the CC. We believe that the Lagrangian dual problem becomes harder to solve when the number of groups and their size increase; that is, it is harder to find feasible solutions that synchronize the CC across all subproblems. We also note that most of the subproblems are solved to optimality for \La{20} and \La{30} (i.e., over 95\% of them), while around 75\% are optimally solved for \La{50}, which could also explain the behaviour for instances with $\numtrips=300$. However, the optimality gaps are usually small (i.e., around 1-5\%) for sub-optimal subproblem solutions. 

Table \ref{tab:lagr-eval} compares the results obtained by the two best Lagrangian-based alternatives in terms of satisfying the CC with the two deterministic alternatives, \mean\ and \percentile, when evaluated over 2000 scenarios. The columns and bold numbers have the same meaning as in Table \ref{tab:b&c-quality}. First, we see that \percentile, \La{20}, and \La{30} have similar objective values, all between 10\% and 15\% increased when compared to \mean. However, there is a significant difference when we compare the percentage of scenarios that satisfy the service requirements. The percentage significantly decreases for \mean\ and \percentile\ as the number of trips increases, which illustrates the lack of reliability guarantees for these approaches as the instances become larger (and, thus, more realistic). In contrast,  \La{20} obtains, on average, a percentage close to our 95\% target, which the percentage is closer to 94\% for \La{30}. Therefore, we can infer that the Lagrangian-based approach can indeed find feasible (or close to feasible) solutions for our CC-MDVSP, especially when the subproblems are small. 

\begin{table}[t] 
  \centering
  \small
  \caption{Solution quality comparison for large-scale instances}
    \begin{tabular}{cc|r|rrr|rrrr}
    \toprule
          &       & \multicolumn{1}{l|}{Obj.Val} & \multicolumn{3}{c|}{\% Obj. Diff.} & \multicolumn{4}{c}{\% Scenarios} \\
    \midrule
    $\numtrips$ & $\numdepots$ & \mean & \percentile & \La{20}  & \La{30}  & \mean & \percentile & \La{20}  & \La{30} \\
    \midrule
    \multirow{3}[1]{*}{200} & 2     & 598.6 & 11.5  & 12.2  & \textbf{10.0} & 31.3  & 91.6  & \textbf{94.5} & \textbf{94.4} \\
          & 3     & 575.1 & \textbf{11.5} & 13.2  & \textbf{11.5} & 31.2  & 91.0  & \textbf{94.3} & 93.7 \\
          & 4     & 573.1 & 12.0  & 12.5  & \textbf{11.4} & 27.6  & 93.8  & \textbf{94.8} & \textbf{95.2} \\
    \midrule
    \multirow{3}[2]{*}{300} & 2     & 843.2 & \textbf{9.0} & 16.6  & 13.8  & 8.3   & 80.3  & \textbf{94.6} & 92.5 \\
          & 3     & 858.4 & \textbf{9.0} & 15.6  & 13.4  & 9.1   & 77.7  & 92.9  & 93.4 \\
          & 4     & 853.0 & \textbf{9.2} & 17.0  & 13.5  & 10.0  & 80.6  & \textbf{94.4} & 93.2 \\
    \midrule
    \multicolumn{2}{c|}{Average} & 716.9 & 10.4  & 14.5  & 12.3  & 19.6  & 85.8  & \textbf{94.2} & 93.7 \\
    \bottomrule
    \end{tabular}%
  \label{tab:lagr-eval}%
\end{table}%

\section{Conclusion} \label{sec:conclusions}

This paper addresses the multi-depot vehicle scheduling problem (MDVSP) by incorporating travel time uncertainty and service reliability into the model. We propose a novel chance-constrained programming model to guarantee on-time performance (OTP) and ensure fair service distribution across different routes. Our study presents two optimization methods: an exact approach using a branch-and-cut procedure and a heuristic method based on Lagrangian decomposition. 
Our experimental results 
demonstrate the effectiveness of our proposed model in achieving reliable OTP. 
For the exact approach, the proposed cut families and enhancement ideas yield significant improvements in computational efficiency and solution quality, while the optimization-based heuristic method provides high-quality practical solutions for larger instances. Our methodological ideas have the potential to be adopted for some other chance-constrained programs with integer recourse.

For MDVSP, our solutions can significantly reduce the number of buses required while maintaining high service reliability. This translates to substantial cost savings for transit agencies, as minimizing the fleet size without compromising OTP is a critical objective. Additionally, our solutions show that it is possible to maintain reliability without resorting to large slack times, which are often added to schedules in practice to account for travel time variability, thereby improving transit speed and overall service attractiveness. 
Equity is another important consideration in our approach. Our model ensures that service improvements are distributed fairly across different routes, supporting the recent equity efforts in transit planning. This proactive consideration of fairness sets our work apart from traditional demand-centric models. Our findings also highlight the potential benefits of the multi-depot scheduling approach in the stochastic setting to achieve more efficient and reliable transit operations, further underscoring the value of our methodology. All in all, addressing some key priorities of transit agencies, namely the goal of delivering a cost-effective, reliable, and equitable service, our work sets a new benchmark in vehicle scheduling and paves the way for future advancements in public transportation planning. 

Future work could explore extending our model to include electric vehicles and other sustainability considerations, further enhancing the applicability of our approach to modern transit systems. Additionally, investigating real-time scheduling adjustments in response to dynamic changes in travel times could further improve service reliability and operational efficiency. Last but not least, the combination of stochastic MDVSP with other planning problems, such as timetabling, would be of high interest. 



\bibliographystyle{informs2014trsc} 
\bibliography{main} 


%
%
%

\newpage
\begin{APPENDICES}
\section{Proofs}
\label{appendix:lpcuts}

\subsection{Proof of Proposition \ref{prop:greedy}} \label{appendix:prop:greedy-proof}

{\sc Proposition \ref{prop:greedy}.} {\it
For a master solution $\bxhat$ and scenario $\s\in \scenarios$,  Algorithm \ref{alg:greedy} returns an optimal solution for \eqref{model:subproblem}, where $\by^*$ is the earliest start times of each timetabled trip.}

\begin{proof}{Proof.}
    First, by construction, a solution of Algorithm \ref{alg:greedy} is feasible for \eqref{model:subproblem} since it satisfies all the constraints in $\subproblem^\s$. 
    To prove that the solution is optimal, we first set the expressing variables to be as large as possible (i.e., $\bu^* = \express$), which is feasible and allows us to set $\by^*$ as the earliest start time for each trip. We use the assumption that the first trip of any bus can start as early as possible, that is, \eqref{ccs:start-depot} is redundant and \eqref{ccs:start-lb} can be tight for the first trip in each bus. We then compute the start time of each trip considering that it will start as early as possible to avoid unnecessary delays, that is the maximum bounds between inequalities \eqref{ccs:start-lb} and \eqref{ccs:start-trip}, as expressed in  \eqref{eq:earliest-start-time}. Since we iterate the trips in order, \eqref{eq:earliest-start-time} is guaranteed to give us the earliest start time of each trip $i \in \trips$ that satisfies all the constraints of the problem. We then set  $v_i^* = 0$ for all trips $i\in \trips$ that present a delay (i.e., $y_i^*> \stime_i + \ub$), and, accordingly, $z^*$ takes the  smallest possible value given $\bv^*$. Notice that the solution is indeed optimal because any $\by^*$ that is not set to its earliest start time could potentially increase the number of $\bv^*$ set to $0$ and, thus, force $z^*$ to be equal to one. 
    \hfill $\square$ 
\end{proof}

\subsection{Proof of Proposition \ref{prop:greedy_mis}} \label{appendix:prop:greedy_mis-proof}

{\sc Proposition \ref{prop:greedy_mis}.} {\it Consider a vehicle schedule $\bus(\xhat)$ and a scenario $\s\in \scenarios$ such that at least one of the service requirements is violated. Then, for any such constraint $\con \in \{\eqref{ccs:service-trips}, \{\eqref{ccs:service-routes}\}_{r\in \routes} \}$, 
Algorithm \ref{alg:cmis-greedy} builds a C-MIS $\scheduleCMIS$.}

\begin{proof}{Proof.}
For simplicity, we illustrate the proof with $\con = \eqref{ccs:service-trips}$ that forces $z=1$. The proof is analogous if we choose any violated inequality in \eqref{ccs:service-routes}. We first argue that $\scheduleCMIS$ is an IS for $\bus(\xhat)$ and scenario $\s$. By construction, Algorithm \ref{alg:cmis-greedy} chooses a subsequence of predecessors for every $i \in \delayMIS$ such that even if the first trip in such subsequence starts as early as possible, trip $i$ will be delayed. Thus, we can guarantee that the chosen subsequences in $\scheduleCMIS$ are sufficient to ensure that the trips in $\delayMIS$ are delayed and, thus, the service requirement is unfulfilled.

To prove that $\scheduleCMIS$ is minimal, we first note that Algorithm \ref{alg:cmis-greedy} explains the delays of trip in $ i \in \delaytrips$ if and only if $ i \in \delayMIS$. By design of Algorithm \ref{alg:cmis-greedy}, we explain the delays of all trips in $\delayMIS$. On the other hand, the predecessor condition of \textit{Step (i)} guarantees that no other delayed trip $j \in \delaytrips\setminus \delayMIS$  appears in the subsequences chosen by Algorithm \ref{alg:cmis-greedy} because $j$ cannot be a predecessor of any $i \in \delayMIS$ and not belong to $\delayMIS$. Thus, by construction, there cannot be a trip $j \in \delaytrips\setminus \delayMIS$ that is also in a pair trip of $\scheduleCMIS$.

Second, removing any trip pair in $\scheduleCMIS$ breaks the subsequence of predecessors that explain the delay of trips in $\delayMIS$. For every $i \in \delayMIS$, Algorithm \ref{alg:cmis-greedy} constructs the minimal subsequence of trips that explain the delay of $i$, so removing any pair in such subsequence will lead to a smaller subsequence where $i$ is not delayed. Since the size of $\delayMIS$ is also minimal, having trips in $\delayMIS$ that can be on-time will satisfy \con, and, thus, make $\scheduleCMIS$ not an IS for the restricted subproblem with \con\ as the only service requirement.
\hfill $\square$
\end{proof}

\subsection{Primal and Dual Model for Strengthen LP} \label{appendix:lpcuts-model}

We now present the LP model with the enhancements described in Section \ref{sec:lp-cuts} to generate Benders cuts, that is, model \ref{model:subproblem-lp}. This model is valid for a given master solution $\bxhat$ and utilizes the optimal greedy solution $(z, \by*,\bv^*, \bu^*)$ from Algorithm \ref{alg:greedy} to close the integrality gap. Recall that the model includes inequality \eqref{lp:delay-trips} to enforce the integrality of $z$ and considers $\bigMotp_i =y^*_i - \stime_i$ for all $i \in \trips$  to force integrality of $\bv$. To ease exposition, we do not replace the proposed values of $\bigMstart_{ij}$ for all $(i,j)\in \compatible$ here and they delay their replacement for the proof in the following section. Also, we assign  $u_i = \express_i$ for all $i \in \trips$, which is optimal for any given solution. Then, the LP in standard form is:
\begin{subequations}
\begin{align}
    \min\; & z \nonumber \\
    \text{s.t.}\; & \sum_{i \in \trips} v_i + \numserviceTrip z \geq \numserviceTrip  , \label{lp:service-trips}\\
    & \sum_{i \in \trips(r)} v_i  +\numserviceRoute_r z \geq \numserviceRoute_r , & \forall r \in \routes, \label{lp:service-routes}\\
    & z + \sum_{i \in \delayMIS} v_i \geq 1, \label{lp:delay-trips}\\
    & y_i - y_j \geq  \dur_j^\s + \travel_{ji}^\s - \express_j -\bigMstart_{ji}\left( 1 - \sum_{k \in \depots} x_{jik} \right), & \forall (j,i)\in \compatible, \label{lp:start-trip}\\
    & y_k \geq \travel_{ki}^\s -\bigMstart_{ki} \left( 1 - x_{kik} \right), & \forall i \in \trips, k \in \depots, \label{lp:start-depot} \\
    & v_i (\ub - y^*_i + \stime_i) - y_i \geq -y^*_i, & \forall i \in \trips, \label{lp:start-ub}\\
    & y _{i} \geq \stime_i - \lb, & \forall i \in \trips,\label{lp:start-lb}  \\
    & z, v _{i}, y_i \geq 0, & \forall i \in \trips. \label{lp:domains}
\end{align}
\end{subequations}

The dual model is detailed in what follows. Variables $\dualtrip \geq 0$, $\dualroute_r\geq 0$ for $r\in \routes$, and $\dualMIS\geq 0 $ are the dual variables associated with constraints \eqref{lp:service-trips}-\eqref{lp:delay-trips}, respectively. Similarly, $\dualstart_{ji}\geq 0$ for $(j,i)\in \compatible$ and $\dualstart_{ki}\geq 0$ for all $i \in \trips, k \in \depots$ are the dual variable associated with constraints \eqref{lp:start-trip} and \eqref{lp:start-depot}, respectively. Lastly, dual variables $\dualOTP_i\geq 0$ for all $i \in \trips$ are associated with \eqref{lp:start-ub} and  $\duallb_i\geq 0$ for all $i \in \trips$ are associated with \eqref{lp:start-lb}.
\begin{subequations}
\begin{align}
    \max\; & \dualtrip\numserviceTrip 
    + \sum_{r\in \routes}\dualroute \numserviceRoute_r 
    + \dualMIS 
    + \lefteqn{\sum_{(j,i)\in \compatible}\dualstart_{ji}\left( \dur_j^\s + \travel^\s_{ji} - \express_i - \bigMstart_{ji}\left( 1 - \sum_{k \in \depots} \xhat_{jik} \right) \right)} \nonumber \\
    & + \sum_{i \in\trips}\sum_{k \in \depots}\dualstart_{ki} \left( \travel_{ki}^\s -\bigMstart_{ki} \left( 1 - \xhat_{kik} \right) \right)
    - \sum_{i \in\trips}\dualOTP_iy^*_i 
    \lefteqn{+ \sum_{i \in\trips}\duallb_i(\stime_i - \lb)} \nonumber \\
    \text{s.t.}\; &
    \sum_{(j,i)\in \compatible} \dualstart_{ji} - \sum_{(i,j)\in \compatible} \dualstart_{ij} + \sum_{k \in \depots} \dualstart_{ki} -\dualOTP_i + \duallb_i \geq 0, & \forall i \in \trips, \label{dual:y} \\
    & \dualtrip + \dualroute_r + \dualOTP_i(\ub - y^*_i + \stime_i) \geq 0, & \forall i \in \trips\setminus\delayMIS, r \in \routes: i \in \trips(r), \\
    &  \dualtrip + \dualroute_r + \dualMIS +\dualOTP_i(\ub - y^*_i + \stime_i) \geq 0,  & \forall i \in \delayMIS, r \in \routes: i \in \trips(r), \\
    & \numserviceTrip\dualtrip + \sum_{r \in \routes} \numserviceRoute_r\dualroute + \dualMIS \geq 1,  \label{dual:z} \\
    & \dualstart_{ji} \geq 0, &\forall (j,i)\in \compatible,\nonumber \\
    & \dualtrip_i, \dualOTP_i, \duallb_i \geq 0, &\forall i \in \trips, \nonumber\\
    & \dualroute_r  \geq 0,  &\forall r \in \routes, \nonumber \\
    &  \dualMIS \geq 0. \nonumber
\end{align}
\end{subequations}

\subsection{Proof of Proposition \ref{prop:lp-and-mis-cuts}} \label{appendix:lpcuts-proof}

{\sc Proposition \ref{prop:lp-and-mis-cuts}.} {\it Consider the master solution $\bxhat$ that violates one of the service requirements $\con \in \{\eqref{ccs:service-trips}, \{\eqref{ccs:service-routes}\}_{r\in \routes} \}$ for a scenario $\s \in \scenarios$, and a $\delayMIS$ constructed using $\con$. Then, there exists a dual optimal solution of \eqref{model:subproblem-lp} such that \eqref{eq:lp-cut} and the C-MIS cut \eqref{eq:mis-cut} for the corresponding $\scheduleCMIS$ are identical.}

\begin{proof}{Proof.} 
In what follows, we use ``hat'' notation to refer to specific values of the dual variables (e.g., $\dualMIShat$ is a specific value assigned to $\dualMIS$). Recall that we are considering a master solution $\bxhat$ that violates one of the service requirements $\con \in \{\eqref{ccs:service-trips}, \{\eqref{ccs:service-routes}\}_{r\in \routes} \}$ for a scenario $\s \in \scenarios$, and a $\delayMIS$ constructed using $\con$.

We start the proof by justifying the structure of the Bender's cut \eqref{eq:lp-cut}. First, recall that constraints \eqref{lp:start-depot} are redundant and just appear in the model for completeness (see Section \ref{sec:problem} for further details). Thus, we can omit them from our analysis or, equivalently, set $\dualstarthat_{ki} = 0$ for all $i \in \trips$, $k \in \depots$. Then, the optimal objective function value of the dual can be written as follows, 
\[
  \sum_{(i,j) \in \compatible}\dualstarthat_{ij}\bigMstart_{ij}\sum_{k \in \depots} \xhat_{ijk} + C =1,
\]
where $C$ is constant concerning all other dual variables. Note that this value is equal to 1 since $\bxhat$ violates one of the service requirements in scenario $\s$. Then, the Bender's cut can be written as follows for master variables $\bx$ and  $z_s$.
\[
z_s \geq \sum_{(i,j) \in \compatible}\dualstarthat_{ij}\bigMstart_{ij}\sum_{k \in \depots} x_{ijk} + C = \sum_{(i,j) \in \compatible}\dualstarthat_{ij}\bigMstart_{ij}\sum_{k \in \depots} x_{ijk} + 1 - \sum_{(i,j) \in \compatible}\dualstarthat_{ij}\bigMstart_{ij}\sum_{k \in \depots} \xhat_{ijk}.
\]
By ordering the terms in the cut, we get \eqref{eq:lp-cut}, that is,
\[
    \sum_{(i,j) \in \compatible}\dualstarthat_{ij}\bigMstart_{ij}\sum_{k \in \depots} x_{ijk} \leq \sum_{(i,j) \in \compatible}\dualstarthat_{ij}\bigMstart_{ij}\sum_{k \in \depots} \xhat_{ijk} - 1 + z_\s,
\]

We now show how to create an optimal dual solution that leads to an C-MIS cut \eqref{eq:mis-cut} with the corresponding $\scheduleCMIS$. We consider a master solution $\bxhat$ that violates the service requirements $\con \in \{\eqref{ccs:service-trips}, \{\eqref{ccs:service-routes}\}_{r\in \routes} \}$ for a scenario $\s\in \scenarios$, the optimal greedy solution $(z, \by^*,\bv^*, \bu^*)$ from Algorithm \ref{alg:greedy}, a set of delay trips $\delayMIS$, and the corresponding $\scheduleCMIS$. 

First, recall that we are considering the following big-M coefficients, which are valid for our analysis since they preserve the optimal solution $(z, \by^*,\bv^*, \bu^*)$. 
\[ \bigMstart_{j,i} = \begin{cases}
y^*_i, &  \forall (j,i) \in \schedule, \\
\text{large number}, & \text{otherwise.} 
\end{cases}\]

To obtain the desired C-MIS cut, we will show that it is possible to create an optimal dual solution with the following dual variable values:
\[
\dualstarthat_{ji} = \begin{cases}
    \displaystyle\frac{1}{y^*_i},  &(j,i)\in \scheduleCMIS, \\
    0,  &(j,i)\in\compatible\setminus\scheduleCMIS. \\
\end{cases} , \quad \forall (j,i) \in \compatible
\]
Since $y_i >0$ for all $i \in \trips$ due to constraint \eqref{lp:start-lb} and our problem assumptions (see Section \ref{sec:problem}), values $\dualstarthat_{ji}$ for all $(j,i) \in \compatible$ are well-defined. Also, note that if we replace these dual variables assignments and the big-M values in \eqref{eq:lp-cut}, we get the desired C-MIS cut. Therefore, the only missing piece is to show that values $\dualstarthat_{ji}$ for all $(j,i)\in\compatible$ lead to an optimal dual solution.

First, we use the complementary slackness conditions to set the values of other dual variables. Since $y_i > 0$ for all $i \in \trips$, dual constraint  \eqref{dual:y} has to be strictly equal to zero. Consider set $\tripsMIS$ to be all trips indices that are involved in the C-MIS $\scheduleCMIS$, that is, $\tripsMIS = \{i \in \trips:\; \exists\; j \text{ such that } (j,i) \in \scheduleCMIS \text{ or } (i,j) \in \scheduleCMIS \}$. Then, 
since we fixed $\dualstarthat_{ji} = 0 $ for all $(j,i)\notin\scheduleCMIS$, we  have that $\dualOTP_i = \duallb_i$ for all $i \notin \tripsMIS$ in order to satisfy \eqref{dual:y}, and we set $\dualOTPhat_i = \duallbhat_i = 0$ in our constructed dual solution. Also, we  set $\dualtriphat = 0$ and $\dualroutehat_r = 0$ 
for all $r \in \routes$.  Lastly, since $z>0$ for any optimal solution due to \eqref{lp:delay-trips}, constraint \eqref{dual:z} has to be strictly equal to one and, thus, $\dualMIShat =1$.

Replacing all the variable assignments into our dual problem and enforcing equality to \eqref{dual:y}, we have the following simplified dual problem. To ease notation, we replace all variable assignments previously described except $\dualstarthat_{ji}$ for $(j,i)\in \scheduleCMIS$. We also replace $ \dur_j^\s + \travel^\s_{ji} - \express_j$ with $y^*_i-y_j^*$ for all $(j,i) \in \scheduleCMIS$ in the objective function since constraint \eqref{lp:start-trip} forces both expressions to be equal in the primal optimal solution given by Algorithm \ref{alg:greedy}. 
\begin{subequations}
\begin{align}
    \max\; & 1
    + \sum_{(j,i)\in \scheduleCMIS}\dualstarthat_{ji}(y^*_i - y_j^*)
    - \sum_{i \in \tripsMIS}\dualOTP_i y_i^*
    \lefteqn{+ \sum_{i \in \tripsMIS}\duallb_i(\stime_i - \lb)} \nonumber \\
    \text{s.t.}\; &
    \sum_{(j,i)\in \scheduleCMIS} \dualstarthat_{ji} - \sum_{(i,j)\in \scheduleCMIS} \dualstarthat_{ij}  -\dualOTP_i + \duallb_i = 0, & \forall i \in \tripsMIS, \label{dual3:y} \\
    &  \dualOTP_i(\stime_i + \ub - y^*_i) \geq 0, & \forall i \in \tripsMIS\setminus\delayMIS, \label{dual3:v_normal}\\
    &  \dualOTP_i \leq \frac{1}{y^*_i -(\stime_i + \ub) },  & \forall i \in \delayMIS, \label{dual3:v_delay} \\
    & \dualOTP_i, \duallb_i \geq 0, &\forall i \in \trips. \nonumber
\end{align}
\end{subequations}

To assign the remaining variables of the models we use \eqref{dual3:y} and show that \eqref{dual3:v_normal} and \eqref{dual3:v_delay} are satisfied. First, note that \eqref{dual3:y} has at most one $\dualstarthat$ term per summation since every trip $i \in \trips$ can have at most one trip as a predecessor and successor, respectively, in any master solution. Thus,  \eqref{dual3:y} can be written as follows for the general case:
\[
\dualstarthat_{ji} - \dualstarthat_{ij'}  -\dualOTP_i + \duallb_i = 0, \quad \forall i \in \tripsMIS,\;  j, j' \in \trips \text{ such that } (j,i), (i,j')\in \scheduleCMIS.
\]
Then, we partition the trips in $\tripsMIS$ in three different subsets: 
\begin{enumerate}
    \item Trips in $\scheduleCMIS$ that are first in a subsequence, that is,
    \[ \tripsMISfirst = \{i \in \tripsMIS: \not\exists\ j \in \tripsMIS \text{ such that } (j,i)\in \scheduleCMIS  \}  \]
    \item Trips in $\scheduleCMIS$ that are last in a subsequence, that is, 
    \[ \tripsMISlast = \{i \in \tripsMIS: \not\exists\ j \in \tripsMIS \text{ such that } (i,j)\in \scheduleCMIS  \}  \]
    \item Trips in $\scheduleCMIS$ that are in the middle of a subsequence, that is, 
    \[ \tripsMISmiddle = \{i \in \tripsMIS: \exists\ j,j' \in \tripsMIS \text{ such that } ,(j',i),(i,j)\in \scheduleCMIS \}  \]
\end{enumerate}

Note that trips in $\tripsMISfirst$ start as early as possible in the solution of Algorithm \ref{alg:greedy}, so constraint \eqref{lp:start-ub}  associated with those indices are loose and, thus, $\dualOTPhat_i = 0$ for all $i \in \tripsMISfirst$. Similarly, all trips in $\tripsMISlast$ and  $\tripsMISmiddle$ start after $\stime_i - \lb$, so the constraints \eqref{lp:start-lb} are loose for such indices and, thus,  $\duallbhat_i = 0$  for all $i \in \tripsMISfirst \cup \tripsMISlast$.Then, replacing these assignments and $\dualstarthat_{ji}$ for $(j,i)\in \scheduleCMIS$ in constraint \eqref{dual3:y}, we get the following values for the remaining dual variables:
\begin{subequations}
\begin{align*}
    & \duallbhat_j = \frac{1}{y^*_i}, & \forall j \in \tripsMISfirst, (j,i)\in \scheduleCMIS, \\
    & \dualOTPhat_j = \frac{1}{y^*_j}, & \forall j \in \tripsMISlast, \\
    & \dualOTPhat_j  = \frac{1}{y^*_j} - \frac{1}{y^*_i}, & \forall j \in \tripsMISmiddle, (j,i)\in \scheduleCMIS.
\end{align*}
\end{subequations}
Note that $\frac{1}{y^*_j} - \frac{1}{y^*_i} \geq 0$ for all $(j,i)\in \scheduleCMIS$ since $j$ is schedule before $i$ in $\bxhat$ and, thus, $y_j^* \leq y_i^*$. 

We now argue that $\dualOTPhat_j$ for each $j \in \tripsMISmiddle\cup\tripsMISlast$ satisfy constraints \eqref{dual3:v_normal}-\eqref{dual3:v_delay}. Note that constraint \eqref{dual3:v_normal} is satisfied since our constructed $\dualOTPhat_i \geq 0$  and $\stime_i + \ub \geq y^*_i$ for all trips $i \in \tripsMIS$ without delay (i.e., not in $\delayMIS$). Constraint \eqref{dual3:v_delay} is satisfied for all $i \in \tripsMISlast$ since 
 \[  \frac{1}{y^*_i} \leq \frac{1}{y_i^* - (\stime_i +\ub)} \]
as $y_i^* > \stime_i +\ub$  and $\stime_i +\ub \geq 0$ hold. Similarly, with some algebraic manipulation, it can be shown that constraint \eqref{dual3:v_delay} is satisfied for all $i \in \tripsMISmiddle\cap \delayMIS$. Therefore, we constructed a feasible dual solution using the desired $\dualstarthat_{ij}$ for all $(i,j) \in \scheduleCMIS$.

Lastly, we replace all of these  variable assignments in the objective function to check that we obtain an optimal solution:
\begin{subequations}
\begin{align*}
    & 1
    + \sum_{(j,i)\in \scheduleCMIS}\dualstarthat_{ji}(y^*_i - y_j^*)
    - \sum_{i \in \tripsMIS}\dualOTPhat_i y_i^*
    + \sum_{i \in \tripsMIS}\duallbhat_i(\stime_i - \lb) \\
    =\ & 1
    + \sum_{(j,i)\in \scheduleCMIS}\dualstarthat_{ji}(y^*_i - y_j^*)
    - \sum_{j \in \tripsMISmiddle}\dualOTPhat_j y_j^* 
    - \sum_{j \in \tripsMISlast}\dualOTPhat_j y_j^* 
    +\sum_{j \in \tripsMISfirst}\duallbhat_j (\stime_j - \lb) \\
    =\ & 1
    + \sum_{(j,i)\in \scheduleCMIS}\left(\frac{y^*_i - y_j^*}{y^*_i} \right)
    - \sum_{j \in \tripsMISmiddle, (j,i)\in \scheduleCMIS }\left( \frac{1}{y^*_i} - \frac{1}{y^*_j} \right) y_i^* 
    - \sum_{i \in \tripsMISlast} \frac{y^*_j}{y^*_j} 
    +\sum_{i \in \tripsMISfirst, (j,i)\in \scheduleCMIS} \frac{\stime_j - \lb}{y^*_i} \\
   =\ &  1
    + \sum_{(j,i)\in \scheduleCMIS}\left( 1 -\frac{y_j^*}{y^*_i} \right)
    - \sum_{j \in \tripsMISmiddle, (j,i)\in \scheduleCMIS }\left(  1 - \frac{y^*_i}{y^*_j} \right)
    - \sum_{i \in \tripsMISlast} 1 
    +\sum_{i \in \tripsMISfirst, (j,i)\in \scheduleCMIS} \frac{y^*_j}{y^*_i} \\
     =\ & 1  
\end{align*}
\end{subequations}

Note that $y^*_j = \stime_j - \lb$ for all $j \in \tripsMISfirst$, which is used going from lines 3 to 4. Also, all the summations in the fourth line cancel each other, which results in an objective value equal to one. Since the dual objective value is equal to the primal objective value, we can guarantee that this dual solution is optimal. Therefore, we have built an optimal dual solution that leads to a C-MIS cut. The validity of the cuts follows from the validity of the C-MIS cut.  
\hfill $\square$
\end{proof}

\subsection{Proof of Proposition \ref{prop:z-continuous}} \label{appendix:prop:z-continuous-proof}

{\sc Proposition \ref{prop:z-continuous}.} {\it The B\&C decomposition scheme with the previously mentioned modification converges in a finite number of iterations and returns an optimal schedule to \eqref{model:chance-saa}.}

\begin{proof}{Proof.}
First, we note that all the proposed cuts introduced in this section have the form of \eqref{master:cuts-added} for any give $s\in \scenarios$, that is,
\[
\sum_{(i,j)\in \arcs} \lambda_{ij}\sum_{k \in \depots} x_{ijk} \leq \lambda_0 - 1 + z_\s,
\]
where $\lambda_{ijk}, \lambda_0 \geq 0$. Due to the master problem constraints \eqref{cc:tripOnce} that enforce each trip to be assigned only once, we can guarantee that $\sum_{(i,j)\in \arcs}\lambda_{ij}\sum_{k \in \depots} x_{ijk} \leq \lambda_0$ for any feasible assignment of $\bx$ in  \eqref{model:master}. Moreover, there exists at least one schedule $\bxhat$ such that $\sum_{(i,j)\in \arcs}\lambda_{ij}\sum_{k \in \depots} \xhat_{ijk} = \lambda_0$, for instance, the schedule that built such constraint. Therefore, there exists a schedule $\bxhat$ that forces variable $z_\s$ to take value one for any generated cut of the form \eqref{master:cuts-added}. 

Second, for a given $\bxhat$, the decomposition scheme now solves \eqref{model:subproblem} for any $\s \in \scenarios$ such that $\zhat_s < 1$. Therefore, our procedure will add the necessary constraints to enforce integrality on the appropriate $\bz$ variables (i.e., the variables associated with scenarios that violate a service requirement) for given candidate schedule $\bxhat$, which are finitely many. Thus, the modified decomposition scheme will converge in a finite number of iterations, that is, at least one iteration for each candidate schedule that satisfies $\mcfset$. Moreover, the returned schedule $\bxhat$ is optimal for \eqref{model:chance-saa} due to the tightness and validity of the generated cuts \eqref{master:cuts-added}. \hfill $\square$
\end{proof}

\hypertarget{appendixB}{}  
\section*{Appendix B: \ Instance Generation}

We now detailed how we generate random instances for the CC-MDVSP. We made three slight modifications to the commonly used approach proposed by \cite{carpaneto1989branch} for the deterministic MDVSP to consider: (i) trips grouped intro routes, (ii) random travel times, and (iii) expressing. All other problem components (e.g., depot capacities, start time of timetable trips, and costs of deadhead trips) are identical to the approach of \cite{carpaneto1989branch}. Lastly, as in \cite{carpaneto1989branch}, we consider all travel times to be integer quantities, thus, we round to the closest integer any fractional number given when sampling random distributions. 

\paragraph{\underline{Routes.}}
As proposed by \cite{carpaneto1989branch}, we first generate a set of random locations in a grid to represent a trip's possible start and end locations. Since the CC-MDVSP trips are grouped into routes, we first randomly assign start and end locations to each route and then assign these locations to trips of such route. In particular, each route $r\in \routes$ has to locations $\loc_1$ and $\loc_2$ and each trip $i \in \trips^r$ can have one of these locations a starting and/or ending point. For instance, trip $i \in \trips^r$ can be a round-trip where $\locstart_i = \locend_i = \loc_1$ (i.e., the trips starts at $\loc_1$, goes to $\loc_2$, and returns to $\loc_1$), while trip $i' \in \trips^r$ ($i \neq i'$) can be a one-direction trip with $\locstart_{i'} = \loc_2$ and $\locend_{i'} = \loc_1$. Since \cite{carpaneto1989branch} distinguishes between short and long trips in terms of duration, we assume that long trips are always round trips while short trips are one-direction trips. We use the same parameters and probabilities used in \cite{carpaneto1989branch} to determine the grid size, the number of generated locations, and the proportion of long and short trips. 

\paragraph{\underline{Random travel times.}} 
The approach proposed by \cite{carpaneto1989branch} used Euclidean distance to compute travel times of deadhead trips (e.g., $\travel_{ij}$ for all $i,j \in \compatible$) and randomly generate durations of timetable trips using uniform distributions with different parameters depending if the trip is long or short. 
We use this same approach and distribution to get the average travel time of deadhead trips and trip duration, respectively. To generate scenarios for our CC, we sample using a log-normal distribution as suggested in the literature (see, for example, the work of \cite{rahman2018analysis}). In particular, each time is sampled using a log-normal distribution where its mean is the average time given by the \cite{carpaneto1989branch} approach, and the standard deviation is 0.2 times the mean. 

\paragraph{\underline{Expressing.}} 
We consider the maximum expressing time to be a proportion of a trip's average duration (i.e., $\bar{\dur}_i$ for all $ i \in \trips$). If a trip $i \in \trips$ has an average duration less than 10, then the maximum expressing is $\express_i = 0$. Otherwise, the maximum expressing is a random integer between 5\% and 10\% of the average duration, that is, $\express_i \in [0.05\bar{\dur}_i, \; 0.01 \bar{\dur}_i].$




\end{APPENDICES}

\end{document}